\documentclass[12pt]{amsart}
\usepackage[all]{xy}
\usepackage{amssymb}
\usepackage{mathtools}
\usepackage{tikz}
\usepackage[titletoc,toc,title]{appendix}
\usepackage{graphics}
\usepackage{float}
\usepackage{graphicx}
\graphicspath{ {images/} }
\usepackage{epstopdf}
\calclayout
\newtheorem{dummy}{anything}[section]
\newtheorem{theorem}[dummy]{Theorem}

\newtheorem{lemma}[dummy]{Lemma}

\newtheorem{corollary}[dummy]{Corollary}

\theoremstyle{definition}%%Change Theoremstyle

\theoremstyle{remark}
\newtheorem{remark}{Remark}

\def\:{\mkern 1.2mu \colon}
\newcommand{\mmatrix}[4]{\left (\vcenter
{\xymatrix@C-2pc@R-2pc{#1&#2\\#3&#4} } \right )}

\begin{document}
\title[Smooth Manifolds With No Real Projective Structure]
{Smooth Manifolds With Infinite Fundamental Group Admitting No Real Projective Structure}


\author{Hat\.{I}ce \c{C}oban}

\address{
Middle East Technical University
\newline\indent
Northern Cyprus Campus, Turkey} \email{hacoban@metu.edu.tr}

\date{\today}

\begin{abstract}\noindent  It is an important question whether it is possible to put a geometry on a given manifold or not. It is well known that any simply connected closed manifold admitting a real projective structure must be a sphere. Therefore, any simply connected manifold\hspace{0.2cm} $M$\hspace{0.2cm} which is not a sphere\hspace{0.2cm} $(\dim M \geq 4)$\hspace{0.2cm} does not admit a real projective structure. Cooper and Goldman gave an example of a $3$-dimensional manifold not admitting a real projective structure and this is the first known example. In this article, by generalizing their work we construct a manifold\hspace{0.2cm} $M^n$\hspace{0.2cm} with the infinite fundamental group\hspace{0.2cm} $\mathbb{Z}_2 \ast \mathbb{Z}_2$,\hspace{0.2cm} for any\hspace{0.2cm} $n\geq 4$,\hspace{0.2cm} admitting no real projective structure.
\end{abstract}
\maketitle

\section{INTRODUCTION}
As stated in Felix Klein's Erlanger program of 1872, the classical\hspace{0.2cm} $(X,G)$\hspace{0.2cm} geometry is the study of the properties of a space\hspace{0.2cm} $X$\hspace{0.2cm} which are invariant under a transitive action of a Lie group\hspace{0.2cm} $G$.\hspace{0.2cm} Although this notion was introduced by Felix Klein, the study is initiated by Ehresmann \cite{ehresmann}. The basic problem is to determine when one can put a certain kind geometric structure on a given manifold and classify such structures up to isomorphism. It is well known that every surface admits a real projective structure and the classification of these structures on surfaces is completely done (\cite{MR1414974}, \cite{MR2170138}).

Thurston's work, starting around the middle 1970's, on geometrization of 3-manifolds is a significant contribution of geometric structures in low dimensional topology (\cite{MR1435975}). A three manifold admitting one of Thurston's geometries except the two of them, which are\hspace{0.2cm} $S^2 \times \mathbb{R}$\hspace{0.2cm} and\hspace{0.2cm} $\mathbb{H}\times \mathbb{R}$,\hspace{0.2cm} has a real projective structure determined uniquely by this structure. In the remaining two cases, the three manifold also has a real projective structure if the group acting on the manifold preserves the orientation on the\hspace{0.2cm} $\mathbb{R}$\hspace{0.2cm} direction (\cite{MR3485333}, \cite{MR1473106}). On the other hand, there are some examples admitting a real projective structure, which is not obtained from Thurston's eight geometries by Benoist's work (\cite{MR2218481}).

It was a conjecture that every three manifold admits a real projective structure. However, D. Cooper and W. Goldman showed that the connected sum of two copies of real projective three spaces does not admit a real projective structure (\cite{MR3485333}).

It is well known that any simply connected manifold admitting a real projective structure is a sphere since the developing map (see p. \pageref{devmap}) must be a covering map.
Since there are many examples of simply connected manifolds which are not spheres in dimension bigger than $3$ (e.g. $\mathbb{CP}^n$), there are many higher dimensional manifolds that do not admit a real projective structure.

The aim of this paper is to construct smooth $n$-dimensional manifolds with the infinite fundamental group\hspace{0.2cm} $\mathbb{Z}_2 \ast \mathbb{Z}_2\hspace{0.2cm}  (n\geq 4)$,\hspace{0.2cm} which do not admit a real projective structure by generalizing Cooper and Goldman's work in \cite{MR3485333}.

\vskip 0.2cm
\noindent{ {ACKNOWLEDGEMENTS.}} I am grateful to my advisor Y{\i}ld{\i}ray Ozan for his support, comments and suggestions on this work.

\section{PRELIMINARIES}

First, we define an\hspace{0.2cm} $(X, G)$\hspace{0.2cm} structure on a manifold\hspace{0.2cm} $M$\hspace{0.2cm} following Ehresmann. Let\hspace{0.2cm} $M$\hspace{0.2cm} be a real analytic manifold modelled on\hspace{0.2cm} $X$\hspace{0.2cm} (there is a local isomorphism between\hspace{0.2cm} $X$\hspace{0.2cm} and\hspace{0.2cm} $M$)\hspace{0.2cm} and\hspace{0.2cm} $G$\hspace{0.2cm} be a Lie group acting on\hspace{0.2cm} $X$\hspace{0.2cm} transitively. Then we say that\hspace{0.2cm} $M$\hspace{0.2cm} has an\hspace{0.2cm} $(X, G)$\hspace{0.2cm} structure or\hspace{0.2cm} $M$\hspace{0.2cm} is an\hspace{0.2cm} $(X, G)$\hspace{0.2cm} manifold. Therefore, an\hspace{0.2cm} $(X, G)$\hspace{0.2cm} manifold has a canonical real analytic structure (see \cite{MR957518}, \cite{MR2275923}, \cite{MR2827816}, \cite{MR2249478} for more information about\hspace{0.2cm} $(X, G)$\hspace{0.2cm} structures).

Let\hspace{0.2cm} $M$\hspace{0.2cm} be any\hspace{0.2cm} $(X, G)$\hspace{0.2cm} manifold and\hspace{0.2cm} $\{(U_i, \phi_i)\}$\hspace{0.2cm} be an atlas on\hspace{0.2cm} $M$\hspace{0.2cm} with transition maps
$$\gamma_{ij}: \phi_i (U_i\cap U_j) \longrightarrow \phi_j (U_i \cap U_j)$$ such that
$$\gamma_{ij} \circ \phi_i = \phi_j.$$
Consider an analytic continuation of\hspace{0.2cm} $\phi_1$\hspace{0.2cm} along a curve\hspace{0.2cm} $\alpha$\hspace{0.2cm} in\hspace{0.2cm} $M$\hspace{0.2cm} beginning in\hspace{0.2cm} $U_1$.\hspace{0.2cm} Inductively, on a component of\hspace{0.2cm} $\alpha \cap U_i$,\hspace{0.2cm} the analytic continuation of\hspace{0.2cm} $\phi_1$\hspace{0.2cm} along\hspace{0.2cm} $\alpha$\hspace{0.2cm} is of the form\hspace{0.2cm} $\gamma \circ \phi_1$,\hspace{0.2cm} where\hspace{0.2cm} $\gamma \in G$.\hspace{0.2cm} Therefore,\hspace{0.2cm} $\phi_1$\hspace{0.2cm} can be analytically continued along every path to\hspace{0.2cm} $\bigcup\limits_{i} U_i$\hspace{0.2cm} on\hspace{0.2cm} $M$.\hspace{0.2cm} It follows that there is a global analytic continuation of\hspace{0.2cm} $\phi_1$\hspace{0.2cm} on the universal cover\hspace{0.2cm} $\widetilde{M}$\hspace{0.2cm} of\hspace{0.2cm} $M$.\hspace{0.2cm} Therefore, one can define a map $$dev: \widetilde{M} \longrightarrow X,$$ which is called a developing map\label{devmap}. The map\hspace{0.2cm} $dev$\hspace{0.2cm} is an immersion and is unique up to composition with elements of\hspace{0.2cm} $G$.\hspace{0.2cm} From the uniqueness property of\hspace{0.2cm} $dev$,\hspace{0.2cm} for any covering transformation\hspace{0.2cm} $\Gamma_{\alpha}$\hspace{0.2cm} of\hspace{0.2cm} $\widetilde{M}$\hspace{0.2cm} over\hspace{0.2cm} $M$,\hspace{0.2cm} there is an element\hspace{0.2cm} $g_{\alpha}$\hspace{0.2cm} of\hspace{0.2cm} $G$\hspace{0.2cm} such that $$dev \circ \Gamma_{\alpha} =g_{\alpha} \circ dev.$$
Since $$dev \circ \Gamma_{\alpha} \circ \Gamma_{\beta} =g_{\alpha} \circ dev \circ \Gamma_{\beta} =g_{\alpha} \circ g_{\beta} \circ dev,$$ it follows that the map
\begin{eqnarray*}
hol:\pi_1(M) & \longrightarrow & G, \\
\alpha & \longmapsto & g_{\alpha}
\end{eqnarray*}
is a homomorphism and called the holonomy of the geometric structure on \hspace{0.2cm}$M$.\hspace{0.2cm} For more details, see \cite{MR1435975}.

The pair\hspace{0.2cm} $(dev, hol)$\hspace{0.2cm} is called a developing pair for the geometric structure\hspace{0.2cm} $(X, G)$.\hspace{0.2cm} A real projective structure on\hspace{0.2cm} $M^{n}$\hspace{0.2cm} is then an\hspace{0.2cm} $(\mathbb{RP}^n, PGL(n+1,\mathbb{R}))$\hspace{0.2cm} structure.

More precisely, \hspace{0.2cm}$M$\hspace{0.2cm} admits a real projective structure if there is a maximal atlas on\hspace{0.2cm} $M$\hspace{0.2cm} with projective coordinate changes. A covering\hspace{0.2cm} $\{U_i\}$\hspace{0.2cm} of\hspace{0.2cm} $M$\hspace{0.2cm} with a family of local diffeomorphisms\hspace{0.2cm} $\phi_i :U_i\longrightarrow  V_i \subset \mathbb{RP}^n$\hspace{0.2cm} is called a projective atlas if the local transformations\hspace{0.2cm} $\phi_j \circ {\phi_i}^{-1} : \phi_i (U_i\cap U_j) \longrightarrow \phi_j (U_i \cap U_j)$\hspace{0.2cm} are projective (i.e. they are restrictions of some elements of the group\hspace{0.2cm} $PGL\big( n+1, \mathbb{R}\big)$).

\begin{figure}[h]
\begin{center}
\scalebox{0.5}{\includegraphics{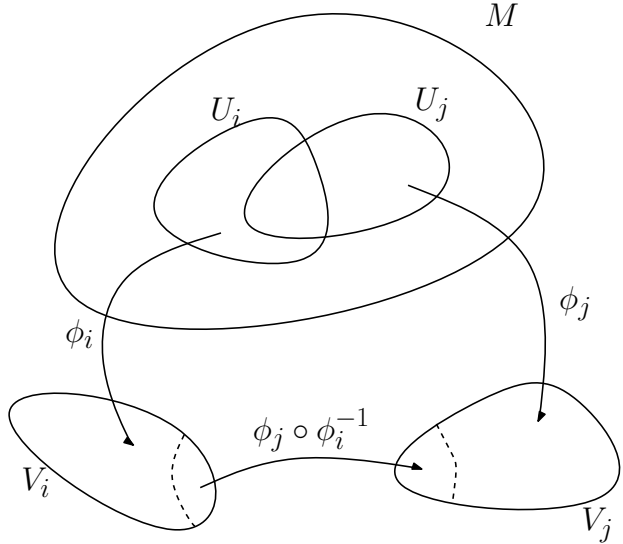}}
\caption{Projective Coordinate Charts}
\end{center}
\end{figure}

\begin{remark}
Let\hspace{0.2cm} $(dev_1, hol_1)$\hspace{0.2cm} and \hspace{0.2cm}$(dev_2, hol_2)$\hspace{0.2cm} be two developing pairs for the same structure. Then  they satisfy the identity\hspace{0.2cm} $dev_2 =g \circ dev_1$,\hspace{0.2cm} for some \hspace{0.2cm}$g \in G$\hspace{0.2cm} and the holonomies are related as\hspace{0.2cm} $hol_2(\beta) =ghol_{1}(\beta)g^{-1}$,\hspace{0.2cm} for any homotopy class\hspace{0.2cm} $[\beta]\in \pi_1 (M)$.
\end{remark}

\begin{theorem}[Ehresmann-Weil-Thurston Principle]\label{EWT}
Let\hspace{0.2cm} $M$\hspace{0.2cm} be an\hspace{0.2cm} $(X, G)$\hspace{0.2cm} manifold with holonomy representation\hspace{0.2cm} $\rho: \pi_1(M)\longrightarrow G$.\hspace{0.2cm} For\hspace{0.2cm} $\rho^{'}$\hspace{0.2cm} sufficiently close to\hspace{0.2cm} $\rho$\hspace{0.2cm} in the space of representations\hspace{0.2cm} $Hom(\pi_1(M), G)$,\hspace{0.2cm} there exists an\hspace{0.2cm} $(X, G)$\hspace{0.2cm} structure on\hspace{0.2cm} $M$\hspace{0.2cm} with holonomy representation\hspace{0.2cm} $\rho^{'}$.
\end{theorem}

The following well known observation is needed in the proof of Theorem \ref{main thm}.
\begin{lemma} \label{covering}
Let\hspace{0.2cm} $X$\hspace{0.2cm} and\hspace{0.2cm} $Y$\hspace{0.2cm} be Hausdorff spaces and\hspace{0.2cm} $f: X\longrightarrow Y$\hspace{0.2cm} be a local homeomorphism. If\hspace{0.2cm} $X$\hspace{0.2cm} is compact and\hspace{0.2cm} $Y$\hspace{0.2cm} is connected then\hspace{0.2cm} $f$\hspace{0.2cm} is a finite sheeted covering map.
\end{lemma}

The following theorem will be needed to study some foliations and the leaf spaces induced by a real projective structure (page \pageref{foli}).

\begin{theorem} (\cite{MR0356087})\label{reebthurston}
Let\hspace{0.2cm} $\mathcal{F}$\hspace{0.2cm} be a codimension one,\hspace{0.2cm} $C^1$,\hspace{0.2cm} transversely oriented foliation of a compact manifold\hspace{0.2cm} $M^n$\hspace{0.2cm} with a compact leaf\hspace{0.2cm} $L$\hspace{0.2cm} such that\hspace{0.2cm} $H^1 (L, \mathbb{R})=0$. Then all leaves of\hspace{0.2cm} $\mathcal{F}$\hspace{0.2cm} are diffeomorphic to\hspace{0.2cm} $L$,\hspace{0.2cm} and the leaves of\hspace{0.2cm} $\mathcal{F}$\hspace{0.2cm} are the fibers of a fibration of\hspace{0.2cm} $M^n$\hspace{0.2cm} over\hspace{0.2cm} $S^1$\hspace{0.2cm} or\hspace{0.2cm} $I$,\hspace{0.2cm} which is an interval. We assume here that if\hspace{0.2cm} $M^n$\hspace{0.2cm} has boundary, then the boundary of\hspace{0.2cm} $M$\hspace{0.2cm} is a union of leaves of\hspace{0.2cm} $\mathcal{F}$.
\end{theorem}

\section{THE MAIN THEOREM} \label{themaintheorem}
In this part, for any\hspace{0.2cm} $n \geq 4$,\hspace{0.2cm} we construct smooth $n$-dimensional manifolds with the fundamental group\hspace{0.2cm} $\mathbb{Z}_2 \ast \mathbb{Z}_2$\hspace{0.2cm} admitting no real projective structure.

Let\hspace{0.2cm} $W$\hspace{0.2cm} be an $m$-dimensional \hspace{0.2cm}($m\geq 3$)\hspace{0.2cm} smooth manifold with\hspace{0.2cm} $\pi_1 (W) \cong \mathbb{Z}_2$\hspace{0.2cm} and\hspace{0.2cm} $S^1$\hspace{0.2cm} denote the unit circle in the complex plane\hspace{0.2cm} $\mathbb{C}$.\hspace{0.2cm} Now let\hspace{0.2cm} $M =\widetilde{W} \times S^1 \big/ <\sigma>$,\hspace{0.2cm} where the action is given by
\begin{eqnarray*}
\sigma :  \widetilde{W} \times S^1 & \longrightarrow & \widetilde{W} \times S^1\\
(p, z) & \longmapsto & (\tau(p), \bar{z})
\end{eqnarray*}
so that\hspace{0.2cm} $<\tau> \cong \mathbb{Z}_2$\hspace{0.2cm} is the Deck transformation group of the universal cover\hspace{0.2cm} $\widetilde{W} \longrightarrow W$\hspace{0.2cm} with\hspace{0.2cm} $\widetilde{W}\big/ <\tau> \cong W.$
Now the universal cover of $M$ is as follows:
$$\widetilde{M}= \widetilde{W} \times \mathbb{R} \longrightarrow \widetilde{W} \times S^1 \longrightarrow \widetilde{W} \times S^1 \big/ <\sigma> \cong M.$$
The induced homomorphism from\hspace{0.2cm} $\sigma$\hspace{0.2cm} on the fundamental group is given as follows:
\begin{eqnarray*}
\sigma_{\sharp}  :  \pi_1(\widetilde{W} \times S^1) & \longrightarrow & \pi_1(\widetilde{W} \times S^1)\\
\mathbb{Z} & \longrightarrow & \mathbb{Z}\\
 1 & \longmapsto & -1 .
\end{eqnarray*}

With an easy observation, one can see the fundamental group of\hspace{0.2cm} $M$\hspace{0.2cm} is as follows:
$$\pi_1 (M) =\mathbb{Z}_2 \ast \mathbb{Z}_2 = <a, b\ |\ a^2 =1,\hspace{0.2cm} b^2 =1> .$$
By using the presentation of the fundamental group of\hspace{0.2cm} $M$,\hspace{0.2cm} we have a short exact sequence
$$1\longrightarrow \pi_1(\widetilde{W}\times S^1)\cong \mathbb{Z} \longrightarrow \pi_1(M) \longrightarrow \mathbb{Z}_2 \longrightarrow 1.$$
The action of\hspace{0.2cm} $\mathbb{Z}_2$\hspace{0.2cm} on the normal subgroup\hspace{0.2cm} $\mathbb{Z}$\hspace{0.2cm} of\hspace{0.2cm} $\pi_1(M)$\hspace{0.2cm} is given by multiplication with\hspace{0.2cm} $-1$.\hspace{0.2cm} Therefore, the fundamental group of\hspace{0.2cm} $M$\hspace{0.2cm} has the following presentation:
$$\pi_1(M) \cong \mathbb{Z} \ltimes \mathbb{Z}_2= <c=ab, a\ |\ a^2=1,\hspace{0.2cm} aca= c^{-1}>.$$
Here is the main result of this paper:

\begin{theorem}\label{main thm}
Let\hspace{0.2cm} $W$\hspace{0.2cm} be an $m$-dimensional ($m\geq 3$) smooth closed manifold with\hspace{0.2cm} $\pi_1 (W) \cong \mathbb{Z}_2$\hspace{0.2cm} and\hspace{0.2cm} $M =\widetilde{W} \times S^1 \big/ <\sigma>$\hspace{0.2cm} as above. We assume that:
\begin{enumerate}
\item Either\hspace{0.2cm} $\widetilde{W}$\hspace{0.2cm} is odd dimensional, or\\
\item $\widetilde{W}$\hspace{0.2cm} is even dimensional and it is not the total space of a sphere bundle over a sphere, where both the base and the fiber are the sphere\hspace{0.2cm} $S^{m/2}$.
\end{enumerate}
Then the manifold\hspace{0.2cm} $M$\hspace{0.2cm} does not admit a real projective structure.
\end{theorem}

\begin{remark}
Note that if\hspace{0.2cm} $m=2$\hspace{0.2cm} and\hspace{0.2cm} $W$\hspace{0.2cm} is a closed surface with\hspace{0.2cm} $\pi_1(W) \cong \mathbb{Z}_2$\hspace{0.2cm} then\hspace{0.2cm} $W =\mathbb{RP}^2$. Thus,\hspace{0.2cm} $\widetilde{W} = S^2$\hspace{0.2cm} and
$$M=S^2 \times S^1 \big/ <\sigma>\hspace{0.1cm} \cong \hspace{0.1cm} \mathbb{RP}^3 \# \mathbb{RP}^3.$$
In other words, our construction does not yield any example other than $\mathbb{RP}^3 \# \mathbb{RP}^3$ in dimension $3$.

Similarly, if\hspace{0.2cm} $m=3$\hspace{0.2cm} and\hspace{0.2cm} $W$\hspace{0.2cm} is a closed $3$-manifold with\hspace{0.2cm} $\pi_1(W)\cong \mathbb{Z}_2$,\hspace{0.2cm} then by the Elliptization Theorem (cf. Theorem 1.12 in \cite{AFW}),\hspace{0.2cm} $W = \mathbb{RP}^3$. Therefore,\hspace{0.2cm} $\widetilde{W} = S^3$\hspace{0.2cm} and thus
$$M = S^3 \times S^1 \big/ <\sigma>\hspace{0.1cm} \cong \hspace{0.1cm} \mathbb{RP}^4 \# \mathbb{RP}^4.$$
\end{remark}

We follow Cooper and Goldman's work closely. Therefore, we omit the proofs of several results, which are analogous to those in \cite{MR3485333}.

Let us take\hspace{0.2cm} $m+1 =n$\hspace{0.2cm} for simplicity. We prove Theorem \ref{main thm} by contradiction. Therefore, we start with the assumption that\hspace{0.2cm} $M$\hspace{0.2cm} admits a real projective structure. Hence, there exists a developing pair\hspace{0.2cm} $(dev, hol)$\hspace{0.2cm} for\hspace{0.2cm} $M$. Before the proof of Theorem \ref{main thm}, we will prove the following lemma (compare with Lemma 4.1 in \cite{MR3485333}).

\begin{lemma}\label{injective}
The map $hol : \pi_1(M) \longrightarrow PGL(n+1, \mathbb{R})$ is injective.
\end{lemma}
\begin{proof}
Suppose not. Then the image of the holonomy is a proper quotient of the infinite dihedral group. This implies that it is finite (\cite{MR0274575}). Let\hspace{0.2cm} $H$\hspace{0.2cm} be the kernel of the homomorphism
$$hol: \pi_1 (M) \longrightarrow PGL(n+1, \mathbb{R}),$$
and\hspace{0.2cm} ${\widetilde{M}}^{'} \longrightarrow M$\hspace{0.2cm} be the covering space corresponding to the subgroup\hspace{0.2cm} $H\leq \pi_1(M)$. Hence, the covering map\hspace{0.2cm} $\widetilde{M}^{'}\longrightarrow M$\hspace{0.2cm} is finite, whose total space  is immersed into\hspace{0.2cm} $\mathbb{RP}^n$\hspace{0.2cm} by the map\hspace{0.2cm} $\varphi: \widetilde{M}^{'} \longrightarrow \mathbb{RP}^n$.\hspace{0.2cm}
Here, the developing map descends to\hspace{0.2cm} $\varphi$\hspace{0.2cm} and the map\hspace{0.2cm} $\varphi$\hspace{0.2cm} is a covering map since\hspace{0.2cm} $\widetilde{M}^{'}$\hspace{0.2cm} is compact.
\begin{displaymath}
\xymatrix {
\widetilde{M} \ar[r]^{dev} \ar[d] &
\mathbb{RP}^n  \\
{\widetilde{M}}^{'} \ar[ur]^{\varphi} \ar[d]  \\
 M }
\end{displaymath}
Thus,\hspace{0.2cm} $\widetilde{M}^{'}$\hspace{0.2cm} is a covering space of\hspace{0.2cm} $\mathbb{RP}^n$. On the other hand,\hspace{0.2cm} $\pi_1(\widetilde{M}^{'})$\hspace{0.2cm} is infinite because $\pi_1(M)$\hspace{0.2cm} is infinite and the covering map\hspace{0.2cm} $\widetilde{M}^{'}\longrightarrow M$\hspace{0.2cm} is finite. Therefore, this gives a contradiction since it is also isomorphic to a subgroup of\hspace{0.2cm} $\pi_1(\mathbb{RP}^n) = \mathbb{Z}_2$.
\end{proof}

\section{PROOF OF THE MAIN THEOREM}
\begin{proof}[Proof of Theorem~\ref{main thm}]
We assume that \hspace{0.2cm}$M$\hspace{0.2cm} admits a real projective structure and thus there exists a developing pair\hspace{0.2cm} $(dev, hol)$
$$dev: \widetilde{M} \longrightarrow \mathbb{RP}^n,$$
where\hspace{0.2cm} $\widetilde{M}$\hspace{0.2cm} is the universal cover of\hspace{0.2cm} $M$\hspace{0.2cm} and
$$hol: \pi_1(M) \longrightarrow PGL(n+1, \mathbb{R}),$$
such that for all\hspace{0.2cm} $\widetilde{m} \in \widetilde{M}$\hspace{0.2cm} and\hspace{0.2cm} $g \in \pi_1(M)$,\hspace{0.2cm} we have $$dev(g\cdot \widetilde{m})= hol(g)\cdot dev(\widetilde{m}).$$

Let\hspace{0.2cm} $[A]$\hspace{0.2cm} and\hspace{0.2cm} $[B]$\hspace{0.2cm} be the images of the generators of the fundamental group\hspace{0.2cm} $\pi_1(M)= \mathbb{Z}_2 \ast \mathbb{Z}_2=<a> \ast <b>$\hspace{0.2cm} under the holonomy map, meaning that
\begin{eqnarray*}
hol: \pi_1 (M) & \longrightarrow & PGL(n+1, \mathbb{R}),\\
a & \longmapsto & hol(a)= [A],\\
b & \longmapsto & hol(b)= [B],
\end{eqnarray*}
where\hspace{0.2cm} $A, B \in GL(n+1, \mathbb{R})$.\hspace{0.2cm} Consider the exact sequence below.
$$1 \longrightarrow \mathbb{Z}  \longrightarrow  \pi_1(M) \cong \mathbb{Z}_2 \ast \mathbb{Z}_2  \longrightarrow  \mathbb{Z}_2  \longrightarrow  1.$$
Here, the infinite cyclic normal subgroup is generated by the product\hspace{0.2cm} $c=ab$.\hspace{0.2cm} For the subgroup of\hspace{0.2cm} $\pi_1(M)$\hspace{0.2cm} generated by\hspace{0.2cm} $a$\hspace{0.2cm} and\hspace{0.2cm} $c^n$\hspace{0.2cm} there is an $n$-fold covering space\hspace{0.2cm}  $M^{(n)} \longrightarrow M$\hspace{0.2cm} and the manifold\hspace{0.2cm} $M^{(n)}$\hspace{0.2cm} is diffeomorphic to\hspace{0.2cm} $M$ \label{covrg}.

\begin{remark}
If a manifold\hspace{0.2cm} $M$\hspace{0.2cm} admits a real projective structure, then any covering space of\hspace{0.2cm} $M$\hspace{0.2cm} admits a real projective structure. In other words, if a covering space of\hspace{0.2cm} $M$\hspace{0.2cm} does not admit a real projective structure then\hspace{0.2cm} $M$\hspace{0.2cm} can not admit a real projective structure.
\end{remark}
Let us take\hspace{0.2cm} $C=AB$.\hspace{0.2cm} After passing to the double cover\hspace{0.2cm} $M^{(2)}$\hspace{0.2cm} of\hspace{0.2cm} $M$,\hspace{0.2cm} we can assume that\hspace{0.2cm} $M$\hspace{0.2cm} has a real projective structure, where\hspace{0.2cm} $A$\hspace{0.2cm} and\hspace{0.2cm} $B$\hspace{0.2cm} are conjugate, which will be explained in the proof of the lemma below.

\begin{lemma}\label{lemma1}
It is possible to arrange that\hspace{0.2cm} $C$\hspace{0.2cm} is diagonalizable over\hspace{0.2cm} $\mathbb{R}$\hspace{0.2cm} with positive eigenvalues.
\end{lemma}

\begin{proof}
First, let us observe that\hspace{0.2cm} $A$\hspace{0.2cm} and\hspace{0.2cm} $B$\hspace{0.2cm} are conjugate on the double cover\hspace{0.2cm} $M^{(2)}$\hspace{0.2cm} of\hspace{0.2cm} $M$:\hspace{0.2cm} Let\hspace{0.2cm} $a^{'}$, $b^{'}$\hspace{0.2cm} and\hspace{0.2cm} $c^{'}$\hspace{0.2cm} be the elements of\hspace{0.2cm} $M^{(2)}$\hspace{0.2cm} such that $$c^{2}=abab=c^{'} =a^{'}b^{'},$$ where \hspace{0.2cm}$a^{'} =a$\hspace{0.2cm} and\hspace{0.2cm} $b^{'} =bab$.\hspace{0.2cm} Therefore, the images of\hspace{0.2cm} $a^{'}$\hspace{0.2cm} and\hspace{0.2cm} $b^{'}$\hspace{0.2cm} are \hspace{0.2cm}$A$\hspace{0.2cm} and\hspace{0.2cm} $BAB$,\hspace{0.2cm}  which are clearly conjugate elements.

Since\hspace{0.2cm} $a^2=1$\hspace{0.2cm} and\hspace{0.2cm} $hol$\hspace{0.2cm} is a homomorphism,\hspace{0.2cm} $[A]^2 \in PGL(n+1, \mathbb{R})$\hspace{0.2cm} is the identity. It follows that after rescaling\hspace{0.2cm} $A$\hspace{0.2cm} we have\hspace{0.2cm} $A^2=\pm Id$,\hspace{0.2cm} thus\hspace{0.2cm} $A$\hspace{0.2cm} is diagonalizable over\hspace{0.2cm} $\mathbb{C}$.\hspace{0.2cm} If\hspace{0.2cm} $A^2= Id$\hspace{0.2cm} then the eigenvalues are\hspace{0.2cm} $\pm 1$.\hspace{0.2cm} Since we are only interested  in\hspace{0.2cm} $[A]$\hspace{0.2cm} we can multiply\hspace{0.2cm} $A$\hspace{0.2cm} with\hspace{0.2cm} $-1$\hspace{0.2cm} and arrange that the eigenvalue\hspace{0.2cm} $-1$\hspace{0.2cm} has multiplicity at most\hspace{0.2cm} $\displaystyle \frac{n+1}{2}$\hspace{0.2cm} (if $n$ is odd) and \hspace{0.2cm}$\displaystyle\frac{n}{2}$\hspace{0.2cm} (if $n$ is even). Otherwise,\hspace{0.2cm} $A^2 = -Id$.\hspace{0.2cm} Depending on the dimension, we have the cases below:

\begin{enumerate}
\item If the dimension\hspace{0.2cm} $n$\hspace{0.2cm} is odd, there exist\hspace{0.2cm} $\displaystyle \frac{n+1}{2} +1$\hspace{0.2cm} possible cases, up to conjugation, for the matrix\hspace{0.2cm} $A$.\hspace{0.2cm}
In the first case\hspace{0.2cm} $A^2 =-Id$\hspace{0.2cm} and the corresponding\hspace{0.2cm} $(n+1) \times (n+1)$\hspace{0.2cm} matrix is
$$\begin{bmatrix}
    0 & 1 & 0 & 0 & \dots  & 0 & 0 \\
    -1 & 0 & 0 & 0 & \dots  & 0 & 0 \\
    0 & 0 & 0 & 1 & \dots & 0 & 0 \\
    0 & 0 & -1 & 0 & \dots & 0 & 0 \\
    \vdots & \vdots & \vdots & \vdots & \ddots & \vdots & \vdots \\
    0 & 0 & 0 & 0 & \dots & 0 & 1 \\
    0 & 0 & 0 & 0 & \dots & -1 & 0
\end{bmatrix}.$$
In the remaining cases\hspace{0.2cm} $A^2 = Id$\hspace{0.2cm} and there are\hspace{0.2cm} $\displaystyle \frac{n+1}{2}$\hspace{0.2cm} possibilities. Along the diagonal there exist only\hspace{0.2cm} $\pm 1$'s\hspace{0.2cm} and all off-diagonal elements are\hspace{0.2cm} $0$.\hspace{0.2cm} The number of\hspace{0.2cm} $-1$\hspace{0.2cm} eigenvalues of each\hspace{0.2cm} $A_i$\hspace{0.2cm} is\hspace{0.2cm} $i$\hspace{0.2cm} and the other eigenvalues are\hspace{0.2cm} $1$,\hspace{0.2cm} where\hspace{0.2cm} $i \in \{1, 2, \dots , \frac{n+1}{2}\}$. \hspace{0.2cm}For example,\hspace{0.2cm} $A_3$\hspace{0.2cm} is as follows:
$$\begin{bmatrix}
    -1 & 0 & 0 & 0 & \dots  & 0 & 0 \\
    0 & -1 & 0 & 0 & \dots  & 0 & 0 \\
    0 & 0 & -1 & 0 & \dots  & 0 & 0 \\
    0 & 0 & 0 & 1 & \dots  & 0 & 0 \\
    \vdots & \vdots & \vdots & \vdots & \ddots  & \vdots & \vdots \\
    0 & 0 & 0 & 0 & \dots & 1 & 0  \\
    0 & 0 & 0 & 0 & \dots & 0 & 1
\end{bmatrix}.$$

\item If the dimension\hspace{0.2cm} $n$\hspace{0.2cm} is even, there exist\hspace{0.2cm} $\displaystyle \frac{n}{2}$\hspace{0.2cm} possible cases for\hspace{0.2cm} $A$.\hspace{0.2cm}
Note that in this case\hspace{0.2cm} $A^2 \neq -Id$\hspace{0.2cm} since\hspace{0.2cm} $A$\hspace{0.2cm} has eigenvalues\hspace{0.2cm} $\pm i$\hspace{0.2cm} and\hspace{0.2cm} $n+1$\hspace{0.2cm} is odd. Hence\hspace{0.2cm} $A^2= Id$\hspace{0.2cm} and thus the possible\hspace{0.2cm} $A_i$\hspace{0.2cm} matrices are similar with the odd dimensional case.
\end{enumerate}
Since\hspace{0.2cm} $A$\hspace{0.2cm} and\hspace{0.2cm} $B$\hspace{0.2cm} are conjugate, there is an element\hspace{0.2cm} $P \in GL(n+1, \mathbb{R})$\hspace{0.2cm} such that\hspace{0.2cm} $B=P A P^{-1}$.\hspace{0.2cm} Then since\hspace{0.2cm} $C=AB$,\hspace{0.2cm} $C =A P A P^{-1}$.\hspace{0.2cm} Changing \hspace{0.2cm}$P$\hspace{0.2cm} is a way to deform the holonomy in the sense of Theorem ~\ref{EWT}.

Define the maps
$$f:GL(n+1, \mathbb{R}) \longrightarrow SL(n+1, \mathbb{R})$$ by
$$f(P)= A P A P^{-1},$$ and
$$g:SL(n+1, \mathbb{R}) \longrightarrow \mathbb{R}^2$$ by
$$g(Q)= (\text{trace}(Q), \text{trace}(Q^2)).$$
Note that these two maps are regular. Choosing an appropriate\hspace{0.2cm} $P$\hspace{0.2cm} depending on\hspace{0.2cm} $A$\hspace{0.2cm} can be done as follows:

\textbf{Case 1:} If\hspace{0.2cm} $A$\hspace{0.2cm} has only one\hspace{0.2cm} $-1$\hspace{0.2cm} eigenvalue then the\hspace{0.2cm} $+1$\hspace{0.2cm} eigenspaces of\hspace{0.2cm} $A$\hspace{0.2cm} and\hspace{0.2cm} $B$\hspace{0.2cm} intersect in a subspace of dimension at least\hspace{0.2cm} $n-1$\hspace{0.2cm} (if the manifold has dimension $n$) for every choice of\hspace{0.2cm} $P$.\hspace{0.2cm} Since\hspace{0.2cm} $C=AB$,\hspace{0.2cm} there is an $(n-1)$-dimensional subspace, on which $C$ is identity and thus $C$ has eigenvalue $1$ with multiplicity at least $(n-1)$. Moreover, by using below\hspace{0.2cm} $P_{t\times t}$\hspace{0.2cm} matrix one can see that trace$(f)$ is nonconstant. If\hspace{0.2cm} $t$\hspace{0.2cm} is odd
$$\text{trace}(f)=t-1+\frac{x+2t-6}{x}$$ and if\hspace{0.2cm} $t$\hspace{0.2cm} is even
$$\text{trace}(f)=\frac{\displaystyle \frac{t^2 -6t+8}{2}x+ \frac{t^3 -10t^2 +28t-32}{4}}{\displaystyle \frac{t-2}{2}x+\displaystyle \frac{t^2 -6t+4}{4}},$$
where $x \in \mathbb{R}$.
Since\hspace{0.2cm} $\text{trace}(f)\neq n+1$,\hspace{0.2cm} there exist two more eigenvalues\hspace{0.2cm} $\lambda$\hspace{0.2cm} and\hspace{0.2cm} $\lambda^{-1}$\hspace{0.2cm} of\hspace{0.2cm} $C$.\hspace{0.2cm} Here we can assume\hspace{0.2cm} $\lambda \neq 1$\hspace{0.2cm} by replacing\hspace{0.2cm} $C$\hspace{0.2cm} to\hspace{0.2cm} $C^2$\hspace{0.2cm} if needed and then clearly\hspace{0.2cm} $\lambda^{-1}\neq 1$.\hspace{0.2cm} It follows that $C$ has eigenvalues $\lambda$, $\lambda^{-1}$ and $1$.

We may take\hspace{0.2cm} $P=(a_{ij})_{t\times t}$,\hspace{0.2cm} $(t=n+1)$\hspace{0.2cm} as follows:

%$$\begin{bmatrix}
%    a_{11} & a_{12} & a_{13} & a_{14}  & \dots  & a_{1(t-1)} & a_{1t} \\
%    a_{21} & a_{22} & a_{23} & a_{24}  & \dots  & a_{2(t-1)} & a_{2t} \\
%    a_{31} & a_{32} & a_{33} & a_{34}  & \dots  & a_{3(t-1)} & a_{3t} \\
%    a_{41} & a_{42} & a_{43} & a_{44}  & \dots  & a_{4(t-1)} & a_{4t} \\
%    \vdots & \vdots & \vdots & \vdots & \ddots  & \vdots & \vdots \\
%    a_{(t-1)1} & a_{(t-1)2} & a_{(t-1)3} & a_{(t-1)4}   & \dots  & a_{(t-1)(t-1)} & a_{(t-1)t}\\
%    a_{t1} & a_{t2} & a_{t3} & a_{t4}   & \dots  & a_{t(t-1)} & a_{tt}
%\end{bmatrix}$$

$\bullet$ If\hspace{0.2cm} $t$\hspace{0.2cm} is even, let
\begin{displaymath}
a_{k1} =a_{1k} = \left\{ \begin{array}{ll}
1, & \textrm{ $k$ is odd},\\
0, & \textrm{$k$ is even},
\end{array} \right.
\end{displaymath}

\begin{displaymath}
a_{kt} = \left\{ \begin{array}{ll}
0, & \textrm{ $k$ is odd},\\
1, & \textrm{$k$ is even},
\end{array} \right.
\end{displaymath}

\begin{displaymath}
a_{tk} = \left\{ \begin{array}{ll}
0, & \textrm{ $k$ is odd},\\
1, & \textrm{$k$ is even and $k \neq 2$},\\
x, & \textrm{ $k=2$},
\end{array} \right.
\end{displaymath}
and the core \hspace{0.2cm}$(t-2)\times (t-2)$\hspace{0.2cm} matrix
$$\begin{bmatrix}
     a_{22} & a_{23} & a_{24}  & \dots  & a_{2(t-1)}  \\
     a_{32} & a_{33} & a_{34}  & \dots  & a_{3(t-1)}  \\
     a_{42} & a_{43} & a_{44}  & \dots  & a_{4(t-1)}  \\
    \vdots & \vdots & \vdots  & \ddots  & \vdots   \\
     a_{(t-1)2} & a_{(t-1)3} & a_{(t-1)4}   & \dots  & a_{(t-1)(t-1)}
\end{bmatrix}$$
is the mirror image of the identity matrix
$$\begin{bmatrix}
     0 & 0 & 0  & \dots & 0 & 1  \\
     0 & 0 & 0  & \dots & 1 & 0  \\
     0 & 0 & 0  & \dots  & 0 & 0 \\
    \vdots & \vdots & \vdots  & \ddots  & \vdots & \vdots  \\
     0 & 1 & 0   & \dots  & 0 & 0\\
     1 & 0 & 0   & \dots  & 0  & 0
\end{bmatrix}.$$

$\bullet$ If\hspace{0.2cm} $t$\hspace{0.2cm} is odd, let
\begin{displaymath}
a_{1k} = \left\{ \begin{array}{ll}
0, & \textrm{ $k\geq 3$ is odd or $k=2$},\\
1, & \textrm{$k$ is even and $k \neq 2$ or $k=1$},
\end{array} \right.
\end{displaymath}

\begin{displaymath}
a_{(k+1)1} =a_{kt}= \left\{ \begin{array}{ll}
0, & \textrm{ $k$ is odd},\\
1, & \textrm{$k$ is even},
\end{array} \right.
\end{displaymath}

\begin{displaymath}
a_{tk} = \left\{ \begin{array}{ll}
0, & \textrm{ $k$ is odd and $k\neq 1$},\\
1, & \textrm{$k$ is even and $k \neq 2$},\\
x, & \textrm{$k=2$},
\end{array} \right.
\end{displaymath}
and the core\hspace{0.2cm} $(t-2)\times (t-2)$\hspace{0.2cm} matrix is the identity matrix.

\textbf{Case 2:} $A$\hspace{0.2cm} has two\hspace{0.2cm} $-1$\hspace{0.2cm} eigenvalues. Then we choose\hspace{0.2cm} $P$\hspace{0.2cm} as follows:

$\bullet$ If\hspace{0.2cm} $t$\hspace{0.2cm} is even, let

$k=t/2$\hspace{0.2cm} and\hspace{0.2cm} $a_{k1}=y$.\hspace{0.2cm} If $k \neq t/2$,\hspace{0.2cm} let
\begin{displaymath}
a_{k1} = \left\{ \begin{array}{ll}
1, & \textrm{ $k$ is odd},\\
0, & \textrm{$k$ is even}.
\end{array} \right.
\end{displaymath}
Also let\quad $a_{12}=y+x,\quad a_{1(t-1)}=y$,\quad $a_{1k}=0$\quad for\quad $3\leq k \leq t-2$\quad
$a_{t2}=x,\\ a_{t(t-1)}=y-x$,\quad $a_{tk}=0$\quad for\quad $3\leq k \leq t-2$.
When\quad $k=(t/2)+1$\quad let\quad $a_{kt}=x$. Otherwise, (i.e. $k\neq (t/2)+1$)
\begin{displaymath}
a_{kt} = \left\{ \begin{array}{ll}
0, & \textrm{ $k$ is odd},\\
1, & \textrm{$k$ is even},
\end{array} \right.
\end{displaymath}
and the core matrix\hspace{0.2cm} $(t-2)\times (t-2)$\hspace{0.2cm} is the identity matrix.

In this case,
\begin{equation*}
\begin{split}
\text{trace}(Q) & = t-4 -\frac{2(-1+x)}{1-x-y+yx-y^2 +x^2}-\frac{(-y-x)(-y+x)}{1-x-y+yx-y^2 +x^2}\\
& -\frac{(-1+x)y}{1-x-y+yx-y^2 +x^2} -\frac{2(y-1)}{1-x-y+yx-y^2 +x^2}\\
& -\frac{-x+xy-y^2+x^2}{1-x-y+yx-y^2 +x^2} - \frac{-y^2-y+yx+x^2}{1-x-y+yx-y^2 +x^2}\\
& -\frac{x(y-1)}{1-x-y+yx-y^2 +x^2} -\frac{(y-x)(y+x)}{1-x-y+yx-y^2 +x^2},
\end{split}
\end{equation*}
where\hspace{0.2cm} $Q=APAP^{-1}$.

Consider the composition below.
$$\mathbb{R}^2  \longrightarrow  GL(n+1, \mathbb{R}) \longrightarrow  SL(n+1, \mathbb{R})  \longrightarrow  \mathbb{R}^2 $$
given by
$$(x, y)   \longrightarrow  P  \longrightarrow  f(P)=A P A P^{-1}  \longrightarrow  g(Q)= (\text{trace}(Q), \text{trace} (Q^2)).$$
Then the determinant of the Jacobian matrix of the composition at\hspace{0.2cm} $(2, 3)$\hspace{0.2cm} is\hspace{0.2cm} $-128$.

For the other cases, we refer the reader to Appendix A.

For each case except Case 1, for a generic\hspace{0.2cm} $P$,\hspace{0.2cm} the dimension of\hspace{0.2cm} $+1$\hspace{0.2cm} eigenspace of\hspace{0.2cm} $C$\hspace{0.2cm} is\hspace{0.2cm} $n-2k$,\hspace{0.2cm} where\hspace{0.2cm} $k$\hspace{0.2cm} is the number of\hspace{0.2cm} $-1$\hspace{0.2cm} eigenvalues of\hspace{0.2cm} $A$.\hspace{0.2cm} We call a matrix\hspace{0.2cm} $P$\hspace{0.2cm} admissible if the number of distinct eigenvalues of\hspace{0.2cm} $C$\hspace{0.2cm} is \hspace{0.2cm}$2k+1$,\hspace{0.2cm} which are\hspace{0.2cm} $1$\hspace{0.2cm} and some pairs\hspace{0.2cm} $\lambda_{1}^{\pm 1}, \lambda_{2}^{\pm 1}, ... , \lambda_{k}^{\pm 1}$,\hspace{0.2cm} $\lambda_i \neq \lambda_{j}^{\pm 1}$\hspace{0.2cm} and\hspace{0.2cm} $\lambda_i \neq 1$. Let\hspace{0.2cm} $E$\hspace{0.2cm} denote the set of non-admissible matrices\hspace{0.2cm} $P$\hspace{0.2cm} in\hspace{0.2cm} $GL(n+1, \mathbb{R})$. Below we will show that\hspace{0.2cm} $E$\hspace{0.2cm} is a proper algebraic set in\hspace{0.2cm} $GL(n+1, \mathbb{R})$\hspace{0.2cm} and thus the set of admissible matrices constitutes an open dense subset in\hspace{0.2cm} $GL(n+1, \mathbb{R})$.

Let\hspace{0.2cm} $T$\hspace{0.2cm} be the set of eigenvalues of\hspace{0.2cm} $C$.\hspace{0.2cm} Since\hspace{0.2cm} $C$\hspace{0.2cm} is conjugate to\hspace{0.2cm} $C^{-1}$,\hspace{0.2cm} there is an involution on\hspace{0.2cm} $T$. If we take\hspace{0.2cm} $P \in E$\hspace{0.2cm} then either some \hspace{0.2cm}$\lambda_i =1$\hspace{0.2cm} or\hspace{0.2cm} $\lambda_i = \lambda^{\pm 1}_j$,\hspace{0.2cm} for some\hspace{0.2cm} $i \neq j$.\hspace{0.2cm} First, assume that some\hspace{0.2cm} $\lambda_i =1$.\hspace{0.2cm} Without loss of generality, let\hspace{0.2cm} $\lambda_k =1$.\hspace{0.2cm} Then
$$\textrm{trace}(C) = m+ \sum_{i=1}^{k-1} (\lambda_i + \lambda_{i}^{-1}).$$
In this case, trace$(C)$, trace$(C^2$), ... ,  trace$(C^k)$ satisfy an algebraic relation. On the other hand, if some\hspace{0.2cm} $\lambda_i = \lambda_{j}^{\pm 1}$,\hspace{0.2cm} for some\hspace{0.2cm} $i \neq j$,\hspace{0.2cm} then again without loss of generality, we may assume that \hspace{0.2cm}$\lambda_{k-1}= \lambda_k$.\hspace{0.2cm} Then
$$\textrm{trace}(C) = m+ \sum_{i=1}^{k-1} a_i (\lambda_i + \lambda_{i}^{-1}),$$ where\hspace{0.2cm} $a_i =1$,\hspace{0.2cm} for\hspace{0.2cm} $1 \leq i \leq k-2$\hspace{0.2cm} and\hspace{0.2cm} $a_{k-1}=2$.\hspace{0.2cm} Hence, again trace$(C)$, trace$(C^2)$, ... , trace$(C^k)$ satisfy an algebraic relation, where\hspace{0.2cm} $1 \leq k \leq (n+1)/2$\hspace{0.2cm} or\hspace{0.2cm} $1 \leq k \leq n/2$\hspace{0.2cm} and\hspace{0.2cm} $\dim\hspace{0.1cm}[g\circ f(E)]= k-1$.\hspace{0.2cm} For example,  if all eigenvalues are\hspace{0.2cm} $\lambda_1$\hspace{0.2cm} and\hspace{0.2cm} $\lambda^{-1}_1$\hspace{0.2cm} then
$$\textrm{trace}(C)= \frac{n+1}{2}(\lambda_1 + \lambda_{1}^{-1}) \textrm{ and  trace}(C^2)= \frac{n+1}{2}(\lambda_{1}^2 + \lambda_{1}^{-2}).$$
Now, trace$(C)$ and trace$(C^2)$ satisfy the following algebraic relation
$$(\textrm{trace}(C))^2 = ((n+1)/2)(\textrm{trace}(C^2))+(n+1)^2/2.$$
Since the determinant of the Jacobian of\hspace{0.2cm} $g \circ f$\hspace{0.2cm} is nonzero at some points, for example\hspace{0.2cm} $(2, 3, ... , k+1)$,\hspace{0.2cm} the image of the map\hspace{0.2cm} $g\circ f$\hspace{0.2cm} contains an open set and thus\hspace{0.2cm} $E$\hspace{0.2cm} is a closed proper subset of\hspace{0.2cm} $GL(n+1, \mathbb{R})$.\hspace{0.2cm} It follows that\hspace{0.2cm} $GL(n+1, \mathbb{R}) \setminus E$\hspace{0.2cm} is open and dense in the Euclidean topology. Therefore, it is possible to perturb\hspace{0.2cm} $P$\hspace{0.2cm} slightly and thus the map\hspace{0.2cm} $hol$\hspace{0.2cm} so that the matrix\hspace{0.2cm} $C$\hspace{0.2cm} is diagonalizable over complex numbers.

With a proper choice of\hspace{0.2cm} $P$,\hspace{0.2cm} it can be arranged that the arguments of complex eigenvalues\hspace{0.2cm} $\lambda_i$\hspace{0.2cm} of\hspace{0.2cm} $C$\hspace{0.2cm} are rational multiples of\hspace{0.2cm} $\pi$.\hspace{0.2cm} Moreover, passing to a finite covering space\hspace{0.2cm} $M^{(n)}$\hspace{0.2cm} of\hspace{0.2cm} $M$\hspace{0.2cm} (see page \pageref{covrg}), we can suppose all eigenvalues of\hspace{0.2cm} $C$\hspace{0.2cm} are real and by passing to a further double cover these eigenvalues can be assumed to be positive.

This concludes the proof of Lemma \ref{lemma1}.
\end{proof}

Hence, we have proved the following lemma for the case\hspace{0.2cm} $A^2 =Id$.

\begin{lemma}
We can arrange that\hspace{0.2cm} $C_i$\hspace{0.2cm} corresponding to\hspace{0.2cm} $A_i$\hspace{0.2cm} so that its eigenvalues are\hspace{0.2cm} $\{\lambda_{i}^{\pm}\}$\hspace{0.2cm} such that\hspace{0.2cm} $\lambda_i > \lambda_{i-1}> ... >\lambda_1 >1$,\hspace{0.2cm} where\hspace{0.2cm} $i \in \{1, ... , \lfloor(n+1)/2\rfloor \}$\hspace{0.2cm} and the remaining eigenvalues of\hspace{0.2cm} $C_i$\hspace{0.2cm} are all\hspace{0.2cm} $1$.
\end{lemma}
When\hspace{0.2cm} $A^2 = -Id$\hspace{0.2cm} the corresponding matrix\hspace{0.2cm} $C$\hspace{0.2cm} has eigenvalues\hspace{0.2cm} $\lambda_1$\hspace{0.2cm} and\hspace{0.2cm} $\lambda^{-1}_1$\hspace{0.2cm} with multiplicities both equal to\hspace{0.2cm} $\displaystyle \frac{n+1}{2}$.

When\hspace{0.2cm} $A^2 =Id$,\hspace{0.2cm} the possible\hspace{0.2cm} $C_i$\hspace{0.2cm} matrices can be arranged depending on the number of\hspace{0.2cm} $-1$\hspace{0.2cm} eigenvalues of\hspace{0.2cm} $A_i$.\hspace{0.2cm} Namely, the number of\hspace{0.2cm} $-1$\hspace{0.2cm} eigenvalues of\hspace{0.2cm} $A$\hspace{0.2cm} determine the number of different\hspace{0.2cm} $\lambda_i$\hspace{0.2cm} eigenvalues of\hspace{0.2cm} $C$.\hspace{0.2cm} For example, if the number of\hspace{0.2cm} $-1$\hspace{0.2cm} eigenvalues of\hspace{0.2cm} $A$\hspace{0.2cm} is\hspace{0.2cm} $2$\hspace{0.2cm} then the corresponding\hspace{0.2cm} $C$\hspace{0.2cm} is
$$\begin{bmatrix}
    \lambda_1 & 0 & 0 & 0 & 0 & \dots  & 0 & 0 \\
    0 & \lambda_2 & 0 & 0 & 0 & \dots  & 0 & 0 \\
    0 & 0 & \lambda^{-1}_1 & 0 & 0 & \dots  & 0 & 0 \\
    0 & 0 & 0 & \lambda^{-1}_2 & 0 & \dots  & 0 & 0 \\
    0 & 0 & 0 & 0 & 1 & \dots & 0 & 0 \\
    \vdots & \vdots & \vdots & \vdots & \vdots & \ddots  & \vdots & \vdots \\
    0 & 0 & 0 & 0 & 0 & \dots & 1 & 0  \\
    0 & 0 & 0 & 0 & 0 & \dots & 0 & 1
\end{bmatrix}.$$
For each matrix\hspace{0.2cm} $C_i$,\hspace{0.2cm} the multiplicity of the eigenvalue\hspace{0.2cm} $\lambda$\hspace{0.2cm} is the same as the multiplicity of\hspace{0.2cm} $\lambda^{-1}$\hspace{0.2cm} since\hspace{0.2cm} $C_{i}$\hspace{0.2cm} is conjugate to\hspace{0.2cm} $C_{i}^{-1}$.

There is a $1$-parameter diagonal subgroup\hspace{0.2cm} $\rho: \mathbb{R} \longrightarrow G\subset PGL(n+1, \mathbb{R})$\hspace{0.2cm} such that\hspace{0.2cm} $\rho(1) =[C]$.\hspace{0.2cm}
The group\hspace{0.2cm} $G$\hspace{0.2cm} is identified with the unique one parameter subgroup containing the cyclic group\hspace{0.2cm} $K$,\hspace{0.2cm} which is generated by\hspace{0.2cm} $C$\hspace{0.2cm} and thus each element of\hspace{0.2cm} $G$\hspace{0.2cm} has real eigenvalues.\hspace{0.2cm} $K$\hspace{0.2cm} is normal in\hspace{0.2cm} $hol(\pi_1 (M))$,\hspace{0.2cm} so\hspace{0.2cm} $G$\hspace{0.2cm} is normalized by\hspace{0.2cm} $hol(\pi_1(M))$.

Let\hspace{0.2cm} $N \longrightarrow M$\hspace{0.2cm} be the double cover corresponding to the subgroup of\hspace{0.2cm} $\pi_1(M)$\hspace{0.2cm} generated by\hspace{0.2cm} $c=ab$.\hspace{0.2cm} Clearly,\hspace{0.2cm} $N \cong \widetilde{W} \times S^1$\hspace{0.2cm} (see Section \ref{themaintheorem}). Let\hspace{0.2cm} $\pi: \widetilde{N} \longrightarrow N$\hspace{0.2cm} be the universal cover of\hspace{0.2cm} $N$.\hspace{0.2cm} Then\hspace{0.2cm} $N$\hspace{0.2cm} has a real projective structure inheriting from\hspace{0.2cm} $M$\hspace{0.2cm} with the same developing map\hspace{0.2cm} $dev_M =dev_N$.\hspace{0.2cm} The image of the holonomy for this projective structure on\hspace{0.2cm} $N$\hspace{0.2cm} is generated by\hspace{0.2cm} $[C]$.

Let\hspace{0.2cm} $z\in gl(n+1, \mathbb{R})$\hspace{0.2cm} be an infinitesimal generator of\hspace{0.2cm} $G$\hspace{0.2cm} such that\hspace{0.2cm} $G= exp(\mathbb{R} \cdot z)$.

Consider the flow $$\Phi :\mathbb{RP}^n \times \mathbb{R} \longrightarrow \mathbb{RP}^n$$ on\hspace{0.2cm} $\mathbb{RP}^n$\hspace{0.2cm} generated by\hspace{0.2cm} $G$,\hspace{0.2cm} which is given by $$\Phi(x, t) = exp(tz)\cdot x,$$ for\hspace{0.2cm} $x\in \mathbb{RP}^n,\hspace{0.2cm} t\in \mathbb{R}$.

Let\hspace{0.2cm} $V$\hspace{0.2cm} be the vector field on\hspace{0.2cm} $\mathbb{RP}^n$,\hspace{0.2cm} the velocity of this flow. Since the vector field is preserved by this flow,\hspace{0.2cm} $V$\hspace{0.2cm} is also preserved by\hspace{0.2cm} $hol(\pi_1 (N))$.\hspace{0.2cm} Hence,\hspace{0.2cm} $V$\hspace{0.2cm} pulls back via the developing map to a vector field\hspace{0.2cm} $dev^{-1}(V)= \widetilde{v}$\hspace{0.2cm} on\hspace{0.2cm} $\widetilde{N}$\hspace{0.2cm} and it is invariant under covering transformations thus covers a vector field\hspace{0.2cm} $\pi(dev^{-1}(V))=v$\hspace{0.2cm} of\hspace{0.2cm} $N$.

In the paper \cite{MR3485333}, the following two lemmas are proved for the $3$-dimensional case (Lemma 4.5 and Lemma 4.6). Moreover, the results are still valid in our case, for any dimension\hspace{0.2cm} $n \geq 4$\hspace{0.2cm} and for any\hspace{0.2cm} $C_i$\hspace{0.2cm} such that\hspace{0.2cm} $1 \leq i \leq \displaystyle \frac{n+1}{2}$\hspace{0.2cm} ($n$ is odd) or\hspace{0.2cm} $1 \leq i \leq \displaystyle \frac{n}{2}$\hspace{0.2cm} ($n$ is even).

\begin{lemma}\label{no source}
$dev(\widetilde{N})$\hspace{0.2cm} does not contain any source or sink.
\end{lemma}

\begin{lemma} \label{periodic}
The flow which is given by the vector field\hspace{0.2cm} $v$\hspace{0.2cm} on\hspace{0.2cm} $N$\hspace{0.2cm} is periodic and\hspace{0.2cm} $N$\hspace{0.2cm} is fibered as a product\hspace{0.2cm} $\widetilde{W} \times S^1$\hspace{0.2cm} by the flowlines.
\end{lemma}

Let\hspace{0.2cm} $X= \mathbb{RP}^n \setminus Z$,\hspace{0.2cm} where\hspace{0.2cm} $Z$\hspace{0.2cm} is the zero set of\hspace{0.2cm} $V$.\hspace{0.2cm} Then\hspace{0.2cm} $X$\hspace{0.2cm} is foliated by flowlines. Let\hspace{0.2cm} $\mathcal{L}$\hspace{0.2cm} be the leaf space of this foliation. Since\hspace{0.2cm} $G$\hspace{0.2cm} is normalized by\hspace{0.2cm} $hol(\pi_1 (M))$\hspace{0.2cm} it follows that this group acts on\hspace{0.2cm} $\mathcal{L}$.\hspace{0.2cm} Since\hspace{0.2cm} $hol(\pi_1(N)) \subset G$\hspace{0.2cm} the action of\hspace{0.2cm} $hol(\pi_1(N))$\hspace{0.2cm} on\hspace{0.2cm} $\mathcal{L}$\hspace{0.2cm} is trivial, the action\hspace{0.2cm} $hol(\pi_1(M))$\hspace{0.2cm} on\hspace{0.2cm} $\mathcal{L}$\hspace{0.2cm} is induced by the involution\hspace{0.2cm} $\sigma$\hspace{0.2cm} by Section \ref{themaintheorem}. Therefore, the holonomy gives an involution on\hspace{0.2cm} $\mathcal{L}$.

Since\hspace{0.2cm} $dev(\widetilde{N}) \subset X$\hspace{0.2cm} there is a map from the leaf space of the induced foliation on\hspace{0.2cm} $\widetilde{N}$\hspace{0.2cm} into\hspace{0.2cm} $\mathcal{L}$.\hspace{0.2cm} The leaf space of\hspace{0.2cm} $\widetilde{N}$\hspace{0.2cm} is\hspace{0.2cm} $\widetilde{W}$\hspace{0.2cm} by Lemma \ref{periodic}. The induced map \label{induced} $$h: \widetilde{W} \longrightarrow \mathcal{L}$$ is a local homeomorphism. Since\hspace{0.2cm} $dev(\widetilde{N}) \subset \mathbb{RP}^n$\hspace{0.2cm} is invariant under\hspace{0.2cm} $hol(\pi_1(M))$\hspace{0.2cm} it follows that\hspace{0.2cm} $h(\widetilde{W}) \subset \mathcal{L}$\hspace{0.2cm} is invariant under involution.

After determining the possible generators\hspace{0.2cm} $[C_i]$\hspace{0.2cm} of\hspace{0.2cm} $hol$,\hspace{0.2cm} we specify the orbit space\hspace{0.2cm} $\mathcal{L}_i$,\hspace{0.2cm} which corresponds to\hspace{0.2cm} $C_i$\hspace{0.2cm} of\hspace{0.2cm} $X= \mathbb{RP}^n \setminus Z$.

To determine the orbit space of\hspace{0.2cm} $C_i$\hspace{0.2cm} we study the zero set of\hspace{0.2cm} $C_i$\hspace{0.2cm} in the following cases.

\textbf{Case 1:}
If the dimension is even\hspace{0.2cm} $n= 2k$\hspace{0.2cm} then there are\hspace{0.2cm} $\displaystyle \frac{n}{2}=k$\hspace{0.2cm} possible cases. Namely, the zero set\hspace{0.2cm} $Z$\hspace{0.2cm} for\hspace{0.2cm} $C_i$\hspace{0.2cm} is the disjoint union of\hspace{0.2cm} $2i$\hspace{0.2cm} points and a linear subspace\hspace{0.2cm} $\mathbb{RP}^{2k-2i}$,\hspace{0.2cm} where\hspace{0.2cm} $1\leq i\leq k$.

For\hspace{0.2cm} $C_1$,\hspace{0.2cm} the zero set consists of one source, one sink and a copy of\hspace{0.2cm} $\mathbb{RP}^{2k-2}$.\hspace{0.2cm} Call these elements as $p_1$: source, $p_2 : \mathbb{RP}^{2k-2}$ and $p_3$: sink. By taking the boundary of a tubular neighborhood of each element of\hspace{0.2cm} $Z$\hspace{0.2cm} in\hspace{0.2cm} $\mathbb{RP}^{2k}$,\hspace{0.2cm} we determine each set of flowlines between any pair of\hspace{0.2cm} $p_i$'s.

\begin{figure}[h]
\begin{center}
\scalebox{0.5}{\includegraphics{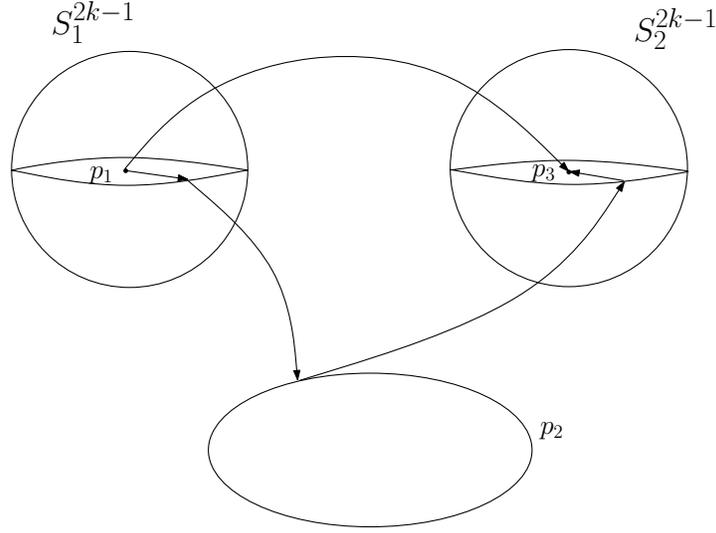}}
\caption{Flowlines for $C_1$}
\label{$C_1$}
\end{center}
\end{figure}

The corresponding flow for\hspace{0.2cm} $C_1$\hspace{0.2cm} is given by
$$\big[x_0{\lambda_1}^t: x_1 {\lambda_1}^{-t}: x_2 : x_3 : ... : x_n \big].$$

If\hspace{0.2cm} $x_1 \neq 0$\hspace{0.2cm} then the flow can be written as $$\Big[ \frac{x_0}{x_1}{\lambda_1}^{2t}: 1: \frac{x_2}{x_1}{\lambda_1}^t : \frac{x_3}{x_1}{\lambda_1}^t : ... : \frac{x_n}{x_1}{\lambda_1}^t \Big].$$
Note that as\hspace{0.2cm} $t\rightarrow - \infty$\hspace{0.2cm} the flow tends to its source, which is\hspace{0.2cm} $\big[0: 1: 0: 0: ... :0\big]$.\\
When\hspace{0.2cm} $x_0 \neq 0$,\hspace{0.2cm} the flow can be written as  $$\Big[ 1: \frac{x_1}{x_0}{\lambda_1}^{-2t}: \frac{x_2}{x_0}{\lambda_1}^{-t}: \frac{x_3}{x_0} {\lambda_1}^{-t}: ... : \frac{x_n}{x_0} {\lambda_1}^{-t}\Big].$$
Similarly, as\hspace{0.2cm} $t\rightarrow + \infty$\hspace{0.2cm} the flow tends to its sink, which is\hspace{0.2cm} $\big[1: 0: 0: 0: ... : 0\big]$.

If we use Euclidean coordinates, there are\hspace{0.2cm} $n$\hspace{0.2cm} parameters in\hspace{0.2cm} $\mathbb{R}^{n}$,\hspace{0.2cm} but one of them is not zero\hspace{0.2cm} $(x_0 \neq 0)$.\hspace{0.2cm} Thus, the flowlines  starting from\hspace{0.2cm} $p_1$\hspace{0.2cm} and leaving\hspace{0.2cm} $S_1$\hspace{0.2cm} from the northern (or the southern) hemisphere go to\hspace{0.2cm} $p_3$\hspace{0.2cm} from the northern (or the southern) hemisphere of\hspace{0.2cm} $S_2$,\hspace{0.2cm} see Figure \ref{$C_1$}.

We consider the source coordinates in Euclidean coordinates as $$\Big( \frac{x_0}{x_1}, \frac{x_2}{x_1}, \frac{x_3}{x_1}, ... , \frac{x_n}{x_1}\Big) \in \mathbb{R}^{n}$$ and the sink coordinates in Euclidean coordinates as $$\Big( \frac{x_1}{x_0}, \frac{x_2}{x_0}, \frac{x_3}{x_0}, ... , \frac{x_n}{x_0}\Big)\in \mathbb{R}^{n}.$$

For the flowlines starting at\hspace{0.2cm} $p_1$\hspace{0.2cm} and leaving\hspace{0.2cm} $S_1$\hspace{0.2cm} from the northern hemisphere we assume\hspace{0.2cm} $\displaystyle \frac{x_0}{x_1} > 0$,\hspace{0.2cm} for\hspace{0.2cm} $x_1 \neq 0$.\hspace{0.2cm} Then\hspace{0.2cm} $\displaystyle \frac{x_1}{x_0} > 0$,\hspace{0.2cm} for\hspace{0.2cm} $x_0 \neq 0$\hspace{0.2cm} and the flowlines go to\hspace{0.2cm} $p_3$\hspace{0.2cm} from the northern hemisphere of\hspace{0.2cm} $S_2$.\hspace{0.2cm} Similarly, when\hspace{0.2cm} $\displaystyle \frac{x_0}{x_1} < 0$,\hspace{0.2cm} for\hspace{0.2cm} $x_0 \neq 0$\hspace{0.2cm} and\hspace{0.2cm} $x_1 \neq 0$\hspace{0.2cm} there is an identification between the southern hemispheres of\hspace{0.2cm} $S_1$\hspace{0.2cm} and\hspace{0.2cm} $S_2$.

Moreover, the flowlines starting from the equator of\hspace{0.2cm} $S_1$\hspace{0.2cm} go to\hspace{0.2cm} $p_2=\mathbb{RP}^{2k-2}$\hspace{0.2cm} and the flowlines starting from\hspace{0.2cm} $p_2$\hspace{0.2cm} go to the equator of\hspace{0.2cm} $S_2$.\hspace{0.2cm} Therefore, the leaf space can be thought as a sphere\hspace{0.2cm} $S^{2k-1}$\hspace{0.2cm} with two disjoint equators. We simply say that\hspace{0.2cm} $S^{2k-1}$\hspace{0.2cm} has a double equator.

The table below describes the subspaces of\hspace{0.2cm} $\mathcal{L}$\hspace{0.2cm} consisting of the flowlines starting at\hspace{0.2cm} $p_i$\hspace{0.2cm} and ending at\hspace{0.2cm} $p_j$,\hspace{0.2cm} for the matrix\hspace{0.2cm} $C_1$.\hspace{0.2cm} The symbol\hspace{0.2cm} 
`$\emptyset$'\hspace{0.2cm} shows that there is no flowline. In the table, we label source points in the upper horizontal line and in the vertical line we label sink points.
\begin{table}[H]
\begin{center}
    \begin{tabular}{| l | l | l | l |}
    \hline
     source/sink & $p_1$ & $p_2$ & $p_3$  \\ \hline
      $p_1$ & $\emptyset$ & $\emptyset$ & $\emptyset$ \\ \hline
      $p_2$ & $S^{2k-2}$ & $\emptyset$ & $\emptyset$ \\ \hline
      $p_3$ & $S^{2k-1}$ & $S^{2k-2}$ & $\emptyset$ \\
    \hline
    \end{tabular}
\bigskip
\caption{The subspaces of the leaf space $\mathcal{L}_1$ for even dimensional case.}
\end{center}
\end{table}
\noindent Note that the table above implies that the leaf space\hspace{0.2cm} $\mathcal{L}_1$\hspace{0.2cm} consists of a copy of\hspace{0.2cm} $S^{2k-1}$\hspace{0.2cm} with a double equator. If a sphere has a double equator, we will denote the sphere as\hspace{0.2cm} $\mathcal{S}$.\hspace{0.2cm} Therefore, the leaf space is\hspace{0.2cm} $\mathcal{L}_1=\mathcal{S}^{2k-1}$.

For\hspace{0.2cm} $C_2$,\hspace{0.2cm} the zero set consists of\hspace{0.2cm} $p_1,\hspace{0.2cm} p_2,\hspace{0.2cm} p_3 = \mathbb{RP}^{2k-4},\hspace{0.2cm} p_4,\hspace{0.2cm} p_5$\hspace{0.2cm} and the leaf space is\hspace{0.2cm} $\mathcal{L}_2=\mathcal{S}^{2k-1} \cup \mathcal{S}^{2k-3}$.\hspace{0.2cm} The corresponding flow is given by $$\big[ x_0 \lambda^{t}_1: x_1 \lambda^{-t}_1: x_2 \lambda^{t}_2: x_3 \lambda^{-t}_2: x_4: x_5: \dots : x_{2k}\big].$$ Similarly, the table below gives a list of subspaces of the space of flowlines for\hspace{0.2cm} $C_2$\hspace{0.2cm} that starts at each\hspace{0.2cm} $p_i$\hspace{0.2cm} and ends at each\hspace{0.2cm} $p_j$,\hspace{0.2cm} $i,j=1, 2, ... , 5.$

\begin{table}[H]
\begin{center}
    \begin{tabular}{| l | l | l | l | l | l |}
    \hline
     source/sink & $p_1$ & $p_2$ & $p_3$ & $p_4$ & $p_5$  \\ \hline
      $p_1$ & $\emptyset$ & $\emptyset$ & $\emptyset$ & $\emptyset$ & $\emptyset$ \\ \hline
      $p_2$ & $S^0$ & $\emptyset$ & $\emptyset$ & $\emptyset$ & $\emptyset$ \\ \hline
      $p_3$ & $S^{2k-3}-S^0$ & $S^{2k-4}$ & $\emptyset$ & $\emptyset$ & $\emptyset$ \\ \hline
      $p_4$ & $S^{2k-2}-S^{2k-3}$ & $S^{2k-3}-S^{2k-4}$ & $S^{2k-4}$ & $\emptyset$ & $\emptyset$ \\ \hline
      $p_5$ & $S^{2k-1}-S^{2k-2}$ & $S^{2k-2}-S^{2k-3}$ & $S^{2k-3}-S^0$ & $S^0$ & $\emptyset$ \\ \hline
    \end{tabular}
\bigskip
\caption{The subspaces of the leaf space $\mathcal{L}_2$ in even dimensional case.}
\end{center}
\end{table}
\noindent In all cases, the subspaces above the diagonal in each table are empty and the nonempty spheres on the antidiagonal have a double equator.

To understand the topology of the space\hspace{0.2cm} $\mathcal{L}$\hspace{0.2cm} clearly, we give an example in dimension $6$. Consider the matrix\hspace{0.2cm} $C_3$\hspace{0.2cm} with different eigenvalues
$$\begin{bmatrix}
    \lambda_1 & 0 & 0 & 0 & 0   & 0 & 0 \\
    0 & \lambda^{-1}_1 & 0 & 0 & 0   & 0 & 0 \\
    0 & 0 & \lambda_2 & 0 & 0  & 0 & 0 \\
    0 & 0 & 0 & \lambda^{-1}_2 & 0  & 0 & 0 \\
    0 & 0 & 0 & 0 & \lambda_3  & 0 & 0 \\
    0 & 0 & 0 & 0 & 0 & \lambda^{-1}_3 & 0  \\
    0 & 0 & 0 & 0 & 0 & 0 & 1
\end{bmatrix}$$
and thus the corresponding leaf space is\hspace{0.2cm} $\mathcal{L}_3=\mathcal{S}^{5} \cup \mathcal{S}^{3} \cup \mathcal{S}^{1}$,\hspace{0.2cm} where\hspace{0.2cm} $\mathcal{S}^{5}=\mathcal{S}^{5}_1$,\hspace{0.2cm} $\mathcal{S}^{3}=\mathcal{S}^{3}_2$,\hspace{0.2cm} $\mathcal{S}^{1}=\mathcal{S}^{1}_3$\hspace{0.2cm} in Table \ref{dimension6}. Moreover, the involution\hspace{0.2cm} $\tau$\hspace{0.2cm} interchanges the spheres symmetric with respect to the vertical line through\hspace{0.2cm} $\mathcal{S}^{5}_1$\hspace{0.2cm} in Table \ref{dimension6}. Indeed,\hspace{0.2cm} $\tau(S^{0}_j)=S^{0}_{7-j}$\hspace{0.2cm} and\hspace{0.2cm} $\tau(\mathcal{S}^{i}_j)=\mathcal{S}^{i}_{7-i-j}$.

\newcommand{\ap}{\ensuremath{\swarrow\,\searrow}}
\setlength{\tabcolsep}{4pt}
\begin{table}[H]
\begin{center}
\begin{tabular}{ccccccccccc}
  &     &     &     &      & $\mathcal{S}^{5}_1$   &      &      &     &    &\\
  &     &     &     &      & \ap &      &      &     &    &\\
  &     &     &     & $\mathcal{S}^{4}_1$    &     &  $\mathcal{S}^{4}_2$   &      &     &    &\\
  &     &     &     & \ap  &     &  \ap &      &     & \\
  &     &     & $\mathcal{S}^{3}_1$   &      & $\mathcal{S}^{3}_2$   &      & $\mathcal{S}^{3}_3$    &     &    &\\
  &     &     & \ap &      & \ap &      & \ap  &     & \\
  &     & $\mathcal{S}^{2}_1$   &     & $\mathcal{S}^{2}_2$    &     &  $\mathcal{S}^{2}_3$   &      & $\mathcal{S}^{2}_4$   &    &\\
  &     &\ap  &     & \ap  &     &  \ap &      & \ap & \\
  & $\mathcal{S}^{1}_1$ &     & $\mathcal{S}^{1}_2$   &      & $\mathcal{S}^{1}_3$   &      & $\mathcal{S}^{1}_4$    &     & $\mathcal{S}^{1}_5$  &\\
  & \ap &     &\ap   &          &\ap      &             &\ap      &              &\ap      &         \\
$S^{0}_1$ &  &  $S^{0}_2$  &    & $S^{0}_3$  &       &  $S^{0}_4$  &   &   $S^{0}_5$  &    &   $S^{0}_6$
\end{tabular}
\bigskip
\caption{The leaf space $\mathcal{L}_3= \mathcal{S}^{5} \cup \mathcal{S}^{3} \cup \mathcal{S}^{1}$. Note that all the spheres in the diagram except the ones in the last row have a double equator.}
\label{dimension6}
\end{center}
\end{table}
\noindent Now, consider the matrix $C_2$ in dimension $6$
$$\begin{bmatrix}
    \lambda_1 & 0 & 0 & 0 & 0   & 0 & 0 \\
    0 & \lambda^{-1}_1 & 0 & 0 & 0   & 0 & 0 \\
    0 & 0 & \lambda_2 & 0 & 0  & 0 & 0 \\
    0 & 0 & 0 & \lambda^{-1}_2 & 0  & 0 & 0 \\
    0 & 0 & 0 & 0 & 1  & 0 & 0 \\
    0 & 0 & 0 & 0 & 0 & 1 & 0  \\
    0 & 0 & 0 & 0 & 0 & 0 & 1
\end{bmatrix}$$
such that the zero set consists of\hspace{0.2cm} $p_1,\hspace{0.2cm} p_2,\hspace{0.2cm} p_3= \mathbb{RP}^2,\hspace{0.2cm} p_4,\hspace{0.2cm} p_5$\hspace{0.2cm} and the corresponding leaf space is\hspace{0.2cm} $\mathcal{L}_2=\mathcal{S}^{5} \cup \mathcal{S}^{3}$.
\setlength{\tabcolsep}{4pt}
\begin{table}[H]
\begin{center}
\begin{tabular}{cccccccccccc}
& & & & & $\mathcal{S}^5$ & & & & &\\
& & & & & \ap   & & & & &\\
& & & & $\mathcal{S}^{4}_1$ & & $\mathcal{S}^{4}_2$ & & & &\\
& & & & \ap       & & \ap       & & & &\\
& & & $\mathcal{S}^{3}_1$ & & $\mathcal{S}^{3}_2$ & &  $\mathcal{S}^{3}_3$ & & &\\
& & & \ap       & & \ap       & & \ap      & & &\\
& & $\swarrow$ & & $S^{2}_1$ & & $S^{2}_2$ & & $\searrow$ & & \\
& $S^{0}_1$ &  &  & &  & &  & &  $S^{0}_2$ &
\end{tabular}
\bigskip
\caption{The leaf space $\mathcal{L}_2= \mathcal{S}^{5} \cup \mathcal{S}^{3}$.}
\end{center}
\end{table}
\noindent Moreover, we have a similar table for each\hspace{0.2cm} $C_i$\hspace{0.2cm} matrix.
In general, for\hspace{0.2cm} $C_i$\hspace{0.2cm} the leaf space becomes\hspace{0.2cm} $\mathcal{L}_{i} = \mathcal{S}^{2k-1} \cup \mathcal{S}^{2k-3} \cup ... \cup \mathcal{S}^{2k-1-2(i-1)}$.

\textbf{Case 2:}
The dimension is odd, let us say\hspace{0.2cm} $n= 2k-1$.\hspace{0.2cm} Then there are\hspace{0.2cm} $\displaystyle \frac{n+1}{2}+1=k+1$\hspace{0.2cm} possible cases.

We get leaf spaces similar to Case 1, for\hspace{0.2cm} $1\leq i\leq k$.\hspace{0.2cm} For example, for\hspace{0.2cm} $C_2$\hspace{0.2cm} the zero set consists of four points and a copy of\hspace{0.2cm} $\mathbb{RP}^{2k-5}$,\hspace{0.2cm} call them\hspace{0.2cm} $p_1$,\hspace{0.2cm} $p_2$,\hspace{0.2cm} $p_3 =\mathbb{RP}^{2k-5}$,\hspace{0.2cm} $p_4$\hspace{0.2cm} and\hspace{0.2cm} $p_5$.

\begin{table}[H]
\begin{center}
    \begin{tabular}{| l | l | l | l | l | l |}
    \hline
     source/sink & $p_1$ & $p_2$ & $p_3$ & $p_4$ & $p_5$  \\ \hline
      $p_1$ & $\emptyset$ & $\emptyset$ & $\emptyset$ & $\emptyset$ & $\emptyset$ \\ \hline
      $p_2$ & $S^0$ & $\emptyset$ & $\emptyset$ & $\emptyset$ & $\emptyset$ \\ \hline
      $p_3$ & $S^{2k-4}-S^0$ & $S^{2k-5}$ & $\emptyset$ & $\emptyset$ & $\emptyset$ \\ \hline
      $p_4$ & $S^{2k-3}-S^{2k-4}$ & $S^{2k-4}-S^{2k-5}$ & $S^{2k-5}$ & $\emptyset$ & $\emptyset$ \\ \hline
      $p_5$ & $S^{2k-2}-S^{2k-3}$ & $S^{2k-3}-S^{2k-4}$ & $S^{2k-4}-S^0$ & $S^0$ & $\emptyset$ \\ \hline
    \end{tabular}
\bigskip
\caption{The subspaces of the leaf space $\mathcal{L}_2$ in odd dimensional case.}
 \label{leaf space}
\end{center}
\end{table}
\noindent The leaf space is\hspace{0.2cm} $\mathcal{L}_2 =\mathcal{S}^{2k-2} \cup \mathcal{S}^{2k-4}$,\hspace{0.2cm} see  Table \ref{leaf space}.

In general, for\hspace{0.2cm} $C_i$\hspace{0.2cm} the zero set consists of\hspace{0.2cm} $p_1,\hspace{0.2cm} p_2, \dots ,p_{i+1}=\mathbb{RP}^{2k-1-2i}, \dots ,p_{2i+1}$\hspace{0.2cm} and the leaf space\hspace{0.2cm} $\mathcal{L}_i =\mathcal{S}^{2k-2} \cup \mathcal{S}^{2k-4} \cup \dots \cup \mathcal{S}^{2k-2i}$,\hspace{0.2cm} for\hspace{0.2cm} $1\leq i < k$\hspace{0.2cm} and if\hspace{0.2cm} $i=k,\hspace{0.2cm}  \mathcal{L}_i =\mathcal{S}^{2k-2} \cup \mathcal{S}^{2k-4} \cup \dots \cup \mathcal{S}^{2} \cup {S}^{0}$.

We have an immersion induced by the developing map (see page \pageref{induced})
$$h: \widetilde{W}^{n-1} \longrightarrow \mathcal{L}^{n-1}.$$
The decomposition of\hspace{0.2cm} $\mathcal{L}$\hspace{0.2cm} contains two $(n-1)$-dimensional open discs, which are\hspace{0.2cm} $D^{n-1}_{+}$\hspace{0.2cm} and\hspace{0.2cm} $D^{n-1}_{-}$,\hspace{0.2cm} see Figure \ref{disc}.

\begin{figure}[h]
\begin{center}
\scalebox{0.6}{\includegraphics{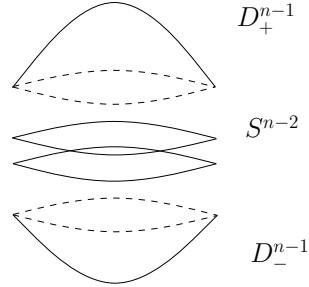}}
\caption{The decomposition of $\mathcal{S}^{n-1}$ in $\mathcal{L}$.}
\label{disc}
\end{center}
\end{figure}

Now, we give some auxiliary lemmas for the proof of Theorem \ref{main thm}. For more detail see \cite{coban}.

\begin{lemma}\label{lemma3}
Assume that\hspace{0.2cm} $D^{n-1}_{\pm}$\hspace{0.2cm} and the map\hspace{0.2cm} $h$\hspace{0.2cm} are as above. Let\hspace{0.2cm} $K \subseteq D^{n-1}_{+}$\hspace{0.2cm} or\hspace{0.2cm} $K \subseteq D^{n-1}_{-}$\hspace{0.2cm} be a closed disc and assume that\hspace{0.2cm} $h^{-1}(K)=H$\hspace{0.2cm} is not empty. Then the restriction of\hspace{0.2cm} $h$\hspace{0.2cm} to the subspace\hspace{0.2cm} $H$,\hspace{0.2cm} $h^{\prime}:H \longrightarrow K$\hspace{0.2cm} is onto and it is a finite sheeted covering.
\end{lemma}
\begin{proof}
Since\hspace{0.2cm} $h^{\prime}:H \longrightarrow K$\hspace{0.2cm} is a submersion,\hspace{0.2cm} $H$\hspace{0.2cm} is a submanifold in\hspace{0.2cm} $\widetilde{W}$\hspace{0.2cm} with boundary. Then the induced map\hspace{0.2cm} $h^{\prime}:H \longrightarrow K$\hspace{0.2cm} is a local homeomorphism. Since local homeomorphisms are open maps,\hspace{0.2cm} $h^{\prime}$\hspace{0.2cm} is open. The map\hspace{0.2cm} $h^{\prime}$\hspace{0.2cm} is also closed. To see this, take a closed subset\hspace{0.2cm} $Y$\hspace{0.2cm} of\hspace{0.2cm} $H$.\hspace{0.2cm} Since\hspace{0.2cm} $H$\hspace{0.2cm} is compact, the subset\hspace{0.2cm} $Y$\hspace{0.2cm} is also compact. The image of\hspace{0.2cm} $Y$\hspace{0.2cm} is compact because\hspace{0.2cm} $h^{\prime}$\hspace{0.2cm} is continuous. Finally, since\hspace{0.2cm} $K$\hspace{0.2cm} is Hausdorff,\hspace{0.2cm} $h^{\prime}(Y)$\hspace{0.2cm} is closed. Therefore, the map\hspace{0.2cm} $h^{\prime}$\hspace{0.2cm} is both open and closed. Then\hspace{0.2cm} $h^{\prime}(H)=K$\hspace{0.2cm} since\hspace{0.2cm} $K$\hspace{0.2cm} is connected. Since\hspace{0.2cm} $H$\hspace{0.2cm} is compact and\hspace{0.2cm} $h^{\prime}:H \longrightarrow K$\hspace{0.2cm} is a local homeomorphism, where both\hspace{0.2cm} $H$\hspace{0.2cm} and\hspace{0.2cm} $K$\hspace{0.2cm} are Hausdorff, by Lemma \ref{covering},\hspace{0.2cm} $h^{\prime}$\hspace{0.2cm} is a covering projection. Moreover, it is finite sheeted since\hspace{0.2cm} $H$\hspace{0.2cm} is compact.
\end{proof}

\begin{lemma}
The image\hspace{0.2cm} $h(\widetilde{W})$\hspace{0.2cm} contains the top dimensional open discs\hspace{0.2cm} $D^{n-1}_{+}$\hspace{0.2cm} and\hspace{0.2cm} $D^{n-1}_{-}$\hspace{0.2cm} in\hspace{0.2cm} $\mathcal{L}$.\hspace{0.2cm} Moreover, when we restrict\hspace{0.2cm} $h$\hspace{0.2cm} to the preimages of these discs, the map\hspace{0.2cm} $h^{-1}(D^{n-1}_{\pm })\longrightarrow D^{n-1}_{\pm}$\hspace{0.2cm} is a finite sheeted covering space.
\end{lemma}
\begin{proof}
Since the map\hspace{0.2cm} $h$\hspace{0.2cm} is a local homeomorphism,\hspace{0.2cm} $h(\widetilde{W})$\hspace{0.2cm} contains at least one point in one of the $(n-1)$-dimensional open discs in\hspace{0.2cm} $\mathcal{L}$.\hspace{0.2cm} By Lemma \ref{lemma3}, if\hspace{0.2cm} $h(\widetilde{W})$\hspace{0.2cm} contains one point of an open disc, it is onto that open disc. Without loss of generality, let us say\hspace{0.2cm} $h(\widetilde{W})$\hspace{0.2cm} contains\hspace{0.2cm} $D^{n-1}_{+}$.\hspace{0.2cm} Assume that\hspace{0.2cm} $h(\widetilde{W})$\hspace{0.2cm} does not contain any point in\hspace{0.2cm} $D^{n-1}_{-}$.\hspace{0.2cm} Then\hspace{0.2cm} $h(\widetilde{W})$\hspace{0.2cm} can not contain a point from the equators\hspace{0.2cm} $S^{n-2}$\hspace{0.2cm} of\hspace{0.2cm} $\mathcal{S}^{n-1}$.\hspace{0.2cm} Because if the image\hspace{0.2cm} $h(\widetilde{W})$\hspace{0.2cm} contained a point from one of the equators\hspace{0.2cm} $S^{n-2}$\hspace{0.2cm} then the neighborhood of that point would have some points from\hspace{0.2cm} $D^{n-1}_{-}$.\hspace{0.2cm} Then in this case,\hspace{0.2cm} $h: \widetilde{W}\longrightarrow D^{n-1}_{+}$\hspace{0.2cm} would be a covering map. Hence,\hspace{0.2cm} $\widetilde{W}$\hspace{0.2cm} would be a disjoint union of open discs, which is a contradiction. In addition,\hspace{0.2cm} $h_{|}: h^{-1}(D^{n-1}_{\pm })\longrightarrow D^{n-1}_{\pm}$\hspace{0.2cm} is a finite sheeted covering space by the above lemma.
\end{proof}
In fact the above lemma implies the following corollary.

\begin{corollary}\label{cor1}
$\widetilde{W}\setminus h^{-1} (D^{n-1}_{\pm})$\hspace{0.2cm} is a nonempty $(n-2)$-dimensional manifold.
\end{corollary}
We will use the following well known fact repeatedly.

\begin{lemma}\label{lemma4}
Let\hspace{0.2cm} $L$\hspace{0.2cm} be an $n$-dimensional connected and simply connected manifold and\hspace{0.2cm} $U\subset L$\hspace{0.2cm} be an open ball, where\hspace{0.2cm} $n\geq 3$.\hspace{0.2cm} Then\hspace{0.2cm} $L \setminus U$\hspace{0.2cm} is connected and simply connected.
\end{lemma}

To proceed further, we consider the three cases of the leaf space\hspace{0.2cm} $\mathcal{L}$.

\textbf{Case 1:} Consider the immersion\hspace{0.2cm} $h: \widetilde{W} \longrightarrow \mathcal{L}_{i} = \mathcal{S}^{n-1} \cup \mathcal{S}^{n-3} \cup ... \cup \mathcal{S}^{n-1-2(i-1)}$,\hspace{0.2cm} for\hspace{0.2cm} $n+1-2i \geq 2$.

Now, we remove the top dimensional open discs namely,\hspace{0.2cm} $D^{n-1}_{+}$\hspace{0.2cm} and\hspace{0.2cm} $D^{n-1}_{-}$\hspace{0.2cm} from the leaf space\hspace{0.2cm} $\mathcal{L}$\hspace{0.2cm} and their preimages from\hspace{0.2cm} $\widetilde{W}$.\hspace{0.2cm} Then the remaining $(n-2)$-dimensional manifold\hspace{0.2cm} $\mathcal{G}^{n-2}=\widetilde{W} \setminus h^{-1}(D^{n-1}_{\pm})$\hspace{0.2cm} is a closed connected manifold by Lemma \ref{lemma4} and Corollary \ref{cor1} and the map
$$\mathcal{G}^{n-2} \longrightarrow S^{n-2} \cup S^{n-2} \cup \mathcal{S}^{n-3} \cup ... \cup \mathcal{S}^{n-1-2(i-1)}$$
is still an immersion. Next, the $(n-2)$-dimensional open discs\hspace{0.2cm} $D^{n-2}$'s\hspace{0.2cm} are removed from\hspace{0.2cm} $\mathcal{L}$\hspace{0.2cm} and their preimages from\hspace{0.2cm} $\mathcal{G}^{n-2}$\hspace{0.2cm} and we get an immersion as follows
$$\mathcal{G}^{n-3} \longrightarrow S^{n-3} \cup S^{n-3} \cup  \mathcal{S}^{n-3} \cup ... \cup \mathcal{S}^{n-1-2(i-1)}.$$
Here,\hspace{0.2cm} $\mathcal{G}^{n-3}$\hspace{0.2cm} is an $(n-3)$-dimensional manifold since the image of\hspace{0.2cm} $\mathcal{G}^{n-2}$\hspace{0.2cm} should contain points from the equators of\hspace{0.2cm} $S^{n-2}$'s.

We continue removing the top dimensional open discs from the leaf space\hspace{0.2cm} $\mathcal{L}$\hspace{0.2cm} and their preimages from the remaining part of\hspace{0.2cm} $\widetilde{W}$\hspace{0.2cm} until we get
$$\mathcal{G}^{n+1-2i} \longrightarrow \mathcal{L}^{n+1-2i}= S^{n+1-2i} \cup S^{n+1-2i} \cup ... \cup S^{n+1-2i} \cup \mathcal{S}^{n+1-2i}.$$
By Lemma \ref{lemma4},\hspace{0.2cm} $\mathcal{G}^{n+1-2i}$\hspace{0.2cm} is still connected and simply connected as long as\\ $n+1-2i\geq 2$.

\textbf{Case 2:} If the dimension of the manifold is\hspace{0.2cm} $2k$,\hspace{0.2cm} for some\hspace{0.2cm} $k \in \mathbb{Z}$\hspace{0.2cm} and\hspace{0.2cm} $i=k$\hspace{0.2cm} then  removing cells as above we finally obtain the following immersion
$$\mathcal{G}^2 \longrightarrow \mathcal{L}^2= S^2 \cup S^2 \cup ... \cup S^2 \cup \mathcal{S}^1.$$ Next, we remove small open discs containing the north and south poles of the $2$-dimensional spheres and one of the equators of\hspace{0.2cm} $\mathcal{S}^1$\hspace{0.2cm} then foliate the complement with circles.

\textbf{Case 3:} If the dimension of the manifold is\hspace{0.2cm} $2k-1$,\hspace{0.2cm} for some\hspace{0.2cm} $k \in \mathbb{Z}$\hspace{0.2cm} and\hspace{0.2cm} $i=k$\hspace{0.2cm} then the immersion analogously will be
$$\mathcal{G}^2 \longrightarrow \mathcal{L}^2= S^2 \cup S^2 \cup ... \cup S^2 \cup \mathcal{S}^2 \cup S^0 .$$ Then we remove small open discs containing the north and south poles of the $2$-dimensional spheres and\hspace{0.2cm} $S^0$\hspace{0.2cm} then foliate the complement with circles.

Note that in all cases above there are foliations on\hspace{0.2cm} $\mathcal{L}^{n+1-2i}$\hspace{0.2cm} with the spheres\hspace{0.2cm} $S^{r}$'s\hspace{0.2cm} ($r=n-2i$\hspace{0.2cm} in Case 1 and\hspace{0.2cm} $r=1$\hspace{0.2cm} in Case 2 and 3) after removing the small open discs containing the north and south poles of each $(n+1-2i)$-sphere in\hspace{0.2cm} $\mathcal{L}^{n+1-2i}$.\hspace{0.2cm} The number of the preimages of these open discs in\hspace{0.2cm} $\mathcal{G}^{n+1-2i}$\hspace{0.2cm} is finite and we remove these open discs from\hspace{0.2cm} $\mathcal{G}^{n+1-2i}$.\hspace{0.2cm} Hence, there is also a foliation on\hspace{0.2cm} $\mathcal{G}^{n+1-2i}$\hspace{0.2cm} with $r$-dimensional manifolds\hspace{0.2cm} $\mathcal{J}^{r}$.\hspace{0.2cm} These $r$-manifolds must be sphere since the sphere\hspace{0.2cm} $S^{r}$\hspace{0.2cm} is closed in the leaf space, its preimage is also closed. Furthermore, the preimage of\hspace{0.2cm} $S^{r}$\hspace{0.2cm} is compact because\hspace{0.2cm} $\mathcal{G}^{n+1-2i}$\hspace{0.2cm} is compact. Hence, the map is a covering by Lemma \ref{covering},\hspace{0.2cm}
$\mathcal{J}^{r} \longrightarrow S^{r}$.

For\hspace{0.2cm} $r =1$,\hspace{0.2cm} after removing the preimages of small open discs containing the north and south poles of each sphere in\hspace{0.2cm} $\mathcal{L}$,\hspace{0.2cm} the remaining part of\hspace{0.2cm} $\widetilde{W}$\hspace{0.2cm} is foliated by 1-dimensional manifolds. These 1-dimensional manifolds are circles since\hspace{0.2cm} $\mathcal{G}^{n+1-2i}$\hspace{0.2cm} is compact and\hspace{0.2cm} $\mathcal{J}^{1} \longrightarrow S^{1}$\hspace{0.2cm} is a covering map. Hence, the remaining part of\hspace{0.2cm} $\widetilde{W}$\hspace{0.2cm} is foliated by circles and hence it is an annulus.

By Theorem \ref{reebthurston}, for\hspace{0.2cm} $r >1$,\hspace{0.2cm} the foliation on\hspace{0.2cm} $\mathcal{G}$\hspace{0.2cm} is\hspace{0.2cm} $S^{r}\times I$,\hspace{0.2cm} where\hspace{0.2cm} $I=[-1, 1]$.\hspace{0.2cm} Hence, the leaf space of this foliation on\hspace{0.2cm} $\mathcal{G}^{n+1-2i}$\hspace{0.2cm} is\hspace{0.2cm} $I= [-1, 1]$. \label{foli}

On the other hand, the quotient space of the leaf space\hspace{0.2cm} $\mathcal{L}^{n+1-2i}$\hspace{0.2cm} is a non-Hausdorff space which is a union of intervals with one extra origin\hspace{0.2cm} $I^{*}= I_1 \cup I_2 \cup ... \cup I_{2i-1} \cup \{ 0^{\prime} \}$.\hspace{0.2cm} The involution interchanges these intervals with each other except the one, which represents the sphere\hspace{0.2cm} $\mathcal{S}^{n+1-2i}$\hspace{0.2cm} and on that sphere it changes the double origins with each other.
There is still an immersion\hspace{0.2cm} $\widetilde{h}: I \longrightarrow I^{*}$\hspace{0.2cm} induced from the immersion\hspace{0.2cm} $h$\hspace{0.2cm} such that\hspace{0.2cm} $\widetilde{h}(\pm 1)$\hspace{0.2cm} are the end points of some of the intervals in \hspace{0.2cm}$I^{*}$.\hspace{0.2cm} Such an immersion is an embedding, whose image contains one interval with only one copy of the origin. Therefore, this gives a contradiction since in this case the immersed image of\hspace{0.2cm} $\widetilde{W}$\hspace{0.2cm} in\hspace{0.2cm} $\mathcal{L}$\hspace{0.2cm} can not be invariant under involution.

Finally, we consider the case, where\hspace{0.2cm} $C_0$\hspace{0.2cm} is the matrix corresponding to\hspace{0.2cm} $A$,\hspace{0.2cm} such that\hspace{0.2cm} $A^2 =-Id$.\hspace{0.2cm} The zero set is the disjoint union of two copies of\hspace{0.2cm} $\mathbb{RP}^{k-1}$\hspace{0.2cm} in\hspace{0.2cm} $\mathbb{RP}^{2k-1}$, $Z= l_1 \cup l_2$,\hspace{0.2cm} where\hspace{0.2cm} $l_i= \mathbb{RP}^{k-1}, i=1,2.$\hspace{0.2cm} Consider the following diagram, where\hspace{0.2cm} $\pi$\hspace{0.2cm} is the universal covering map.
\begin{displaymath}
\xymatrix {
\tilde{N}=\widetilde{W}\times \mathbb{R} \ar[r]^{dev} \ar[d]^{\pi} &
\mathbb{RP}^{2k-1}  \\
N=\widetilde{W}\times S^1
}
\end{displaymath}
Then\hspace{0.2cm} $dev^{-1}(l_i)$\hspace{0.2cm} is invariant under\hspace{0.2cm} $\pi_1(N)\cong \mathbb{Z}$-action and\hspace{0.2cm} $\alpha_i$\hspace{0.2cm} is an $(k-1)$-dimensional submanifold of\hspace{0.2cm} $N$,\hspace{0.2cm} where\hspace{0.2cm} $\alpha_i =dev^{-1}(l_i)/ \mathbb{Z} =\pi(dev^{-1}(l_i)) \subseteq N$.\hspace{0.2cm} Therefore, $$\alpha_i \longrightarrow l_i$$ is a covering map. Now, we have two cases:
\begin{enumerate}
\item $\alpha_1 \cup \alpha_2= \emptyset$.\hspace{0.2cm} It means that\hspace{0.2cm} $dev(\widetilde{N})$\hspace{0.2cm} is empty. Therefore, we can use Lemma \ref{no source} and Lemma \ref{periodic} in this case also.

Note that each flowline starts at\hspace{0.2cm} $l_1$\hspace{0.2cm} and ends at\hspace{0.2cm} $l_2$.\hspace{0.2cm} We consider the boundary of a tubular neighborhood of\hspace{0.2cm} $\mathbb{RP}^{k-1}$\hspace{0.2cm} in\hspace{0.2cm} $\mathbb{RP}^{2k-1}$,\hspace{0.2cm} which is the total space of an\hspace{0.2cm} $S^{k-1}$\hspace{0.2cm} bundle over\hspace{0.2cm} $\mathbb{RP}^{k-1}$.\hspace{0.2cm} Since there exists a unique flowline passing through any point of the total space of this bundle, the leaf space\hspace{0.2cm} $\mathcal{L}$\hspace{0.2cm} is that total space. (Note that if\hspace{0.2cm} $(k-1)$\hspace{0.2cm} is even,\hspace{0.2cm} $\mathbb{RP}^{k-1}$\hspace{0.2cm} is nonorientable and thus the bundle is nontrivial.) The immersion\hspace{0.2cm} $h= \widetilde{W} \longrightarrow \mathcal{L}$,\hspace{0.2cm} induced from the developing map\hspace{0.2cm} $dev: \widetilde{M} \longrightarrow \mathbb{RP}^n$,\hspace{0.2cm} is a covering map. Now consider the diagram below.

\begin{figure}[h]
\begin{center}
\scalebox{0.5}{\includegraphics{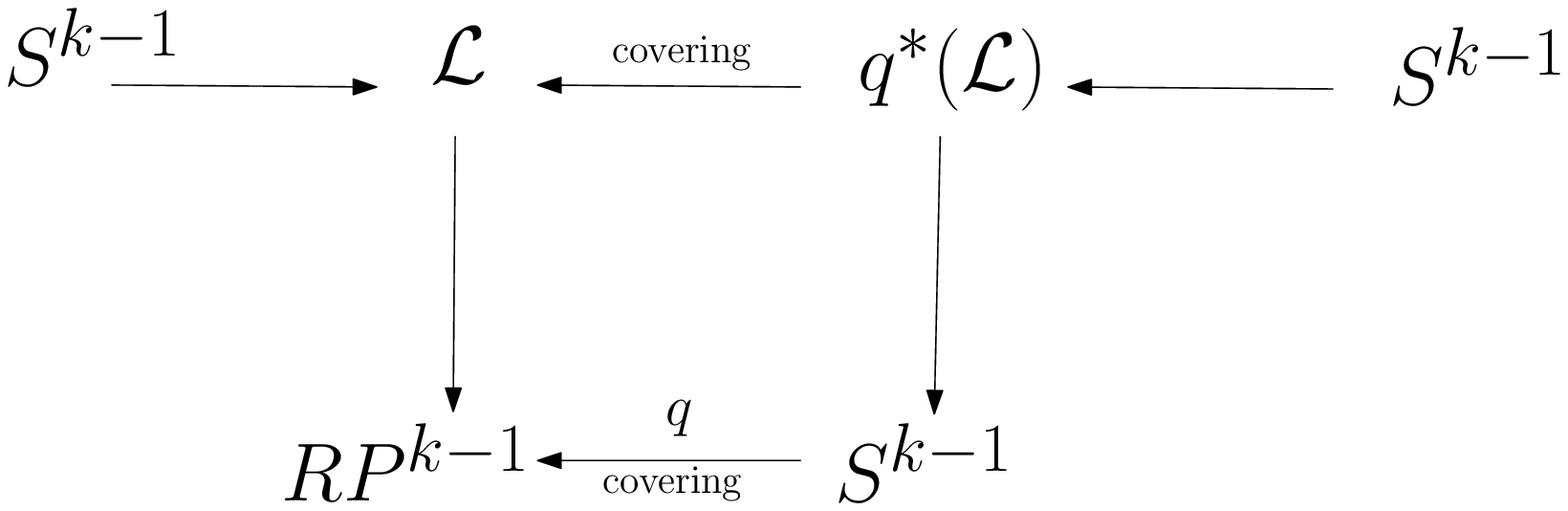}}
\caption{$q^{*}(\mathcal{L})$ is an $S^{k-1}$ bundle over $S^{k-1}$.}
\label{bundle}
\end{center}
\end{figure}
\noindent Since\hspace{0.2cm} $k \geq 3$,\hspace{0.2cm} $\pi_1(q^{*}(\mathcal{L}))=0$,\hspace{0.2cm} so\hspace{0.2cm} $\widetilde{W}$\hspace{0.2cm} and\hspace{0.2cm} $q^{*}(\mathcal{L})$\hspace{0.2cm} are two simply connected coverings of\hspace{0.2cm} $\mathcal{L}$.\hspace{0.2cm} Hence, we have a homeomorphism\hspace{0.2cm} $q^{*}(\mathcal{L})\cong \widetilde{W}$.

Therefore,\hspace{0.2cm} $\widetilde{W}$\hspace{0.2cm} is the total space of an\hspace{0.2cm} $S^{k-1}$\hspace{0.2cm} bundle over\hspace{0.2cm} $S^{k-1}$,\hspace{0.2cm} see Figure \ref{bundle}.

Since by assumption\hspace{0.2cm} $\widetilde{W}$\hspace{0.2cm} is not the total space of an\hspace{0.2cm} $S^{k-1}$\hspace{0.2cm} bundle over\hspace{0.2cm} $S^{k-1}$\hspace{0.2cm} (in the statement of Theorem ~\ref{main thm}),\hspace{0.2cm} $h$\hspace{0.2cm} is not a covering map and this is a contradiction.

\item $\alpha_1 \cup \alpha_2 \neq \emptyset$.\hspace{0.2cm} In this case, Lemma \ref{no source} does not hold and we present an alternative argument as follows:

Without loss of generality, assume that\hspace{0.2cm} $\alpha_1$\hspace{0.2cm} is nonempty. Let\hspace{0.2cm} $\phi$\hspace{0.2cm} be the closure of a flowline of\hspace{0.2cm} $v$\hspace{0.2cm} with one endpoint on\hspace{0.2cm} $\alpha_1$.\hspace{0.2cm} $\phi$\hspace{0.2cm} is a compact $1$-submanifold of\hspace{0.2cm} $N$\hspace{0.2cm} because its preimage in\hspace{0.2cm} $\widetilde{N}$\hspace{0.2cm} maps into a closed invariant interval in\hspace{0.2cm} $\mathbb{RP}^{2k-1}$\hspace{0.2cm} with one endpoint in each\hspace{0.2cm} $l_i$.\hspace{0.2cm} Hence, the other endpoint of\hspace{0.2cm} $\phi$\hspace{0.2cm} is in\hspace{0.2cm} $\alpha_2$,\hspace{0.2cm} which is also necessarily nonempty.

To show that\hspace{0.2cm} $\alpha_1$\hspace{0.2cm} is connected, take a component\hspace{0.2cm} $\gamma$\hspace{0.2cm} of\hspace{0.2cm} $\alpha_1$.\hspace{0.2cm} Let\hspace{0.2cm} $U$\hspace{0.2cm} be the tubular neighborhood of\hspace{0.2cm} $\gamma$\hspace{0.2cm} in\hspace{0.2cm} $N$.\hspace{0.2cm} $dev(\pi^{-1}(\gamma)) \subset l_1$\hspace{0.2cm} and actually they are equal. Hence,\hspace{0.2cm} $dev(\pi^{-1}(U))$\hspace{0.2cm} contains a neighborhood of\hspace{0.2cm} $l_1$.\hspace{0.2cm} Thus,\hspace{0.2cm} $U$\hspace{0.2cm} contains the total space\hspace{0.2cm} $\Upsilon$\hspace{0.2cm} of an\hspace{0.2cm} $S^{k-1}$\hspace{0.2cm} bundle over\hspace{0.2cm} $\mathbb{RP}^{k-1}$\hspace{0.2cm} transverse to the flow and bounds a small neighborhood of\hspace{0.2cm} $\gamma$.\hspace{0.2cm} Since\hspace{0.2cm} $U$\hspace{0.2cm} is preserved by the flow it follows that\hspace{0.2cm} $U= \Upsilon\times \mathbb{R}$.\hspace{0.2cm} The boundary of\hspace{0.2cm} $U$\hspace{0.2cm} in\hspace{0.2cm} $N$\hspace{0.2cm} is contained in\hspace{0.2cm} $\alpha_1 \cup \alpha_2$.\hspace{0.2cm} Therefore,\hspace{0.2cm} $\alpha_1$ and $\alpha_2$\hspace{0.2cm} are both connected and\hspace{0.2cm} $N= \alpha_1 \cup U \cup \alpha_2$\hspace{0.2cm} (this argument is analogous to the one in \cite{MR3485333}, p.8).

Since\hspace{0.2cm} $\alpha_i \longrightarrow l_i$,\hspace{0.2cm} for\hspace{0.2cm} $i=1,2$\hspace{0.2cm} is a covering map, there are two possibilities for\hspace{0.2cm} $\alpha_i$,\hspace{0.2cm} which are\hspace{0.2cm} $S^{k-1}$\hspace{0.2cm} and\hspace{0.2cm} $\mathbb{RP}^{k-1}$.

If\hspace{0.2cm} $\alpha_i =S^{k-1}$\hspace{0.2cm} then the boundary of the neighborhood of\hspace{0.2cm} $\alpha_i$\hspace{0.2cm} is\hspace{0.2cm} $\partial \nu(\alpha_i) =S^{k-1}\widetilde{\times} S^{k-1}$,\hspace{0.2cm} which is the total space of a sphere bundle over a sphere. Since\hspace{0.2cm} $k \geq 3$,\hspace{0.2cm} the homotopy exact sequence implies that\hspace{0.2cm} $\pi_1( \nu(\alpha_i))$\hspace{0.2cm} is trivial.\hspace{0.2cm} $N$\hspace{0.2cm} can be written as\hspace{0.2cm} $N= \nu(\alpha_1) \cup \nu(\alpha_2)$,\hspace{0.2cm} where the two neighborhoods are glued along their boundaries via a diffeomorphism. Finally, by Van Kampen's theorem $$\pi_1(N)\cong \pi_1(\nu(\alpha_1))\ast \pi_1(\nu(\alpha_2))/ L,$$
where\hspace{0.2cm} $L$\hspace{0.2cm} is the normal subgroup corresponding to the kernel of the homomorphism $$\Phi : \pi_1(\nu(\alpha_1))\ast \pi_1(\nu(\alpha_2)) \longrightarrow \pi_1(N).$$ However, this gives a contradiction because\hspace{0.2cm} $\pi_1(N) \cong \mathbb{Z}$.

If\hspace{0.2cm} $\alpha_i= \mathbb{RP}^{k-1}$\hspace{0.2cm} then\hspace{0.2cm} $\partial \nu(\alpha_i) =S^{k-1}\widetilde{\times} \mathbb{RP}^{k-1}$.\hspace{0.2cm} Similarly,
$$\pi_1(N)\cong \pi_1(\nu(\alpha_1))\ast \pi_1(\nu(\alpha_2))/ L$$
and\hspace{0.2cm} $\pi_1(\nu(\alpha_i)) \cong \mathbb{Z}_2$.\hspace{0.2cm} Therefore, we get
$$\mathbb{Z}\cong \mathbb{Z}_2 \ast \mathbb{Z}_2 / L.$$
However, this is not possible since it is a well known fact that\hspace{0.2cm} $\mathbb{Z}_2 \ast \mathbb{Z}_2$\hspace{0.2cm} has no normal subgroup whose quotient is equal to\hspace{0.2cm} $\mathbb{Z}$.\hspace{0.2cm} Therefore, this gives a contradiction.
\end{enumerate}
This finishes the proof.
\end{proof}

\section{An Obstruction To The Existence Of Real Projective Structures}

In this section, we will give an obstruction to obtain examples of manifolds with the infinite fundamental group\hspace{0.2cm} $\mathbb{Z}$\hspace{0.2cm} admitting no real projective structure.

General properties of Pontryagin classes give the following theorem.
\begin{theorem}\label{smale}
If there is an immersion\hspace{0.2cm} $M^{n-1}\longrightarrow \mathbb{R}^n$,\hspace{0.2cm} where\hspace{0.2cm} $M$\hspace{0.2cm} is an orientable manifold then the Pontryagin classes\hspace{0.2cm} $p_i (M^{n-1})$\hspace{0.2cm} are all two torsion, for\hspace{0.2cm} $i\geq 1$.
\end{theorem}

\begin{theorem}\label{thm43}
Let\hspace{0.2cm} $M^n$\hspace{0.2cm} be a simply connected manifold which does not admit any immersion into\hspace{0.2cm} $\mathbb{R}^{n+1}$.\hspace{0.2cm} Then\hspace{0.2cm} $M \times S^1$\hspace{0.2cm} does not have any real projective structure.
\end{theorem}
\begin{proof}
Assume that\hspace{0.2cm} $M \times S^1$\hspace{0.2cm} admits a real projective structure. Then there exists a developing map such that
\begin{displaymath}
\xymatrix{
M \times \mathbb{R} \ar[d]  \ar[r]^{dev} & \mathbb{RP}^{n+1} \\
M \times S^1 }
\end{displaymath}

Consider the following diagram.
\begin{displaymath}
\xymatrix{
 & S^{n+1} \ar[d] \\
M^{n} \ar[ur] \ar[r]^{dev} & \mathbb{RP}^{n+1}  }
\end{displaymath}

Since the map\hspace{0.2cm} $M \longrightarrow \mathbb{RP}^{n+1}$\hspace{0.2cm} is an immersion and the double cover\\ $S^{n+1} \longrightarrow \mathbb{RP}^{n+1}$\hspace{0.2cm} is a local diffeomorphism,\hspace{0.2cm} $M \longrightarrow S^{n+1}$\hspace{0.2cm} is also an immersion. Moreover,
$$M \longrightarrow S^{n+1} \setminus \{p\} =\mathbb{R}^{n+1}$$ is an immersion where\hspace{0.2cm} $p$\hspace{0.2cm} is a point in\hspace{0.2cm} $S^{n+1}$,\hspace{0.2cm} which is not in the image of\hspace{0.2cm} $M$.
However, this yields a contradiction.
\end{proof}

\textbf{Example:} Let\hspace{0.2cm} $M = \mathbb{CP}^2$.\hspace{0.2cm} The first Pontryagin class of\hspace{0.2cm} $\mathbb{CP}^2$\hspace{0.2cm} is\hspace{0.2cm} $p_1 =c_{1}^2 -2c_2$,\hspace{0.2cm} where\hspace{0.2cm} $c_i$'s\hspace{0.2cm} are Chern classes, for\hspace{0.2cm} $i=1, 2$.\hspace{0.2cm} Then
$$p_1= c_{1}^2 -2c_2 =9-2.3 =3.$$
Hence\hspace{0.2cm} $p_1$\hspace{0.2cm} is not a torsion class.\hspace{0.2cm} By Theorem \ref{smale}, there is no immersion\hspace{0.2cm} $\mathbb{CP}^2 \longrightarrow \mathbb{R}^5$\hspace{0.2cm} and it contradicts to the existence of the developing map. Therefore,\hspace{0.2cm} $\mathbb{CP}^2 \times S^1$\hspace{0.2cm} does not have a real projective structure.

\begin{theorem}
Assume that\hspace{0.2cm} $W^{n-1}$\hspace{0.2cm} and\hspace{0.2cm} $M$\hspace{0.2cm} as in Theorem \ref{main thm}. Assume further that the universal cover\hspace{0.2cm} $\widetilde{W}$\hspace{0.2cm} of\hspace{0.2cm} $W$\hspace{0.2cm} does not admit an immersion into\hspace{0.2cm} $\mathbb{R}^n$.\hspace{0.2cm} Then\hspace{0.2cm} $M$\hspace{0.2cm} has no real projective structure.
\end{theorem}

\begin{proof}
Assume on the contrary that\hspace{0.2cm} $M$\hspace{0.2cm} has a real projective structure. Then the universal cover\hspace{0.2cm} $\widetilde{W}\times \mathbb{R}$\hspace{0.2cm} of\hspace{0.2cm} $M$\hspace{0.2cm} has a real projective structure and thus the developing map\hspace{0.2cm} $dev: \widetilde{W}\times \mathbb{R} \longrightarrow \mathbb{RP}^n$\hspace{0.2cm} provides an immersion of\hspace{0.2cm} $\widetilde{W}$\hspace{0.2cm} into\hspace{0.2cm} $\mathbb{R}^n$.\hspace{0.2cm} This finishes the proof.
\end{proof}

\begin{remark}
Note that since\hspace{0.2cm} $S^{n-1}$\hspace{0.2cm} has an immersion into\hspace{0.2cm} $\mathbb{R}^n$,\hspace{0.2cm} the above theorem does not imply that\hspace{0.2cm} $\mathbb{RP}^n \# \mathbb{RP}^n$\hspace{0.2cm} can not have a real projective structure.
\end{remark}

\begin{appendix}
\section{\\Choosing An Appropriate $P$ Depending On The Matrix $A$}
In this section, we continue choosing an appropriate\hspace{0.2cm} $P$\hspace{0.2cm} for\hspace{0.2cm} $A$\hspace{0.2cm} and calculate trace$(Q)$ to say that the determinant of the Jacobian matrix at some points is nonzero by considering the following composition:
$$\mathbb{R}^2  \longrightarrow  GL(n+1, \mathbb{R}) \longrightarrow  SL(n+1, \mathbb{R})  \longrightarrow  \mathbb{R}^2 $$
given by
$$(x, y) \longmapsto P \longmapsto  f(P)=A P A P^{-1} \longmapsto g(Q)= (\text{trace}(Q), \text{trace}(Q^2)).$$

\textbf{Case 2:} $A$\hspace{0.2cm} has two\hspace{0.2cm} $-1$\hspace{0.2cm} eigenvalues. Then we choose\hspace{0.2cm} $P$\hspace{0.2cm} as follows:

$\bullet$ If\hspace{0.2cm} $t$\hspace{0.2cm} is odd,

set\hspace{0.2cm} $k=(t-1)/2$\hspace{0.2cm} and\hspace{0.2cm} $a_{k1}=y$.\hspace{0.2cm} If\hspace{0.2cm} $k\neq (t-1)/2$,\hspace{0.2cm} let
\begin{displaymath}
a_{k1} = \left\{ \begin{array}{ll}
1, & \textrm{ $k$ is odd},\\
0, & \textrm{$k$ is even},
\end{array} \right.
\end{displaymath}
$a_{t2}=x,\hspace{0.2cm} a_{t(t-1)}=y-x$,\hspace{0.2cm} $a_{tk}=0$,\hspace{0.2cm} for $3\leq k \leq t-2$,\hspace{0.2cm}
$a_{12}=y+x,\hspace{0.2cm} a_{1(t-1)}=y$.\hspace{0.2cm} If\hspace{0.2cm} $t\neq 5$,\hspace{0.2cm} take\hspace{0.2cm} $a_{1((t+3)/2)}= a_{1((t-1)/2)}=1$;\hspace{0.2cm} otherwise,\hspace{0.2cm} $a_{1k}=0$,\hspace{0.2cm} for\hspace{0.2cm} $3\leq k \leq t-2$\hspace{0.2cm} and if\hspace{0.2cm} $t=5$\hspace{0.2cm} then\hspace{0.2cm} $a_{13}=1.$\hspace{0.2cm}
When\hspace{0.2cm} $k=((t+1)/2)+1$\hspace{0.2cm} let\hspace{0.2cm} $a_{kt}=x$.\hspace{0.2cm} Otherwise, (i.e. $k\neq ((t+1)/2)+1$)
\begin{displaymath}
a_{kt} = \left\{ \begin{array}{ll}
0, & \textrm{ $k$ is odd},\\
1, & \textrm{$k$ is even},
\end{array} \right.
\end{displaymath}
and the core matrix\hspace{0.2cm} $(t-2)\times (t-2)$\hspace{0.2cm} is the identity matrix.

If\hspace{0.2cm} $A$\hspace{0.2cm} has two\hspace{0.2cm} $-1$\hspace{0.2cm} eigenvalues and\hspace{0.2cm} $t= 9$\hspace{0.2cm} then we choose \hspace{0.2cm}$P_{t\times t}$\hspace{0.2cm} as below.\label{matris4}
$$\begin{bmatrix}
    1 & y+x & 0 & 1 & 0 & 1  & 0  & y & 0 \\
    0 & 1 & 0 & 0 & 0 & 0  & 0 & 0  & 1\\
    1 & 0 & 1 & 0 & 0  & 0 & 0 & 0 & 0\\
    y & 0 & 0 & 1 & 0 & 0  & 0 & 0 & 1 \\
    1 & 0 & 0 & 0 & 1 & 0  & 0 & 0 & 0 \\
    0 & 0 & 0 & 0 & 0 & 1  & 0 & 0 & x\\
    1 & 0 & 0 & 0 & 0 & 0  & 1 & 0 & 0 \\
    0 & 0 & 0 & 0 & 0 & 0  & 0 & 1 & 1\\
    1 & x & 0 & 0 & 0 & 0  & 0 & y-x & 0
\end{bmatrix}.$$

If $(t-1)/2$ is even then
\begin{equation*}
\begin{split}
\text{trace}(Q) & = t-6-\frac{2y}{1+y+2x+y^2}-\frac{-y-x}{1+y+2x+y^2}-\frac{1+y^2}{1+y+2x+y^2}\\
& -\frac{-1-2x-y^2+yx}{1+y+2x+y^2} -\frac{y(1+x)}{1+y+2x+y^2}-\frac{-y^2}{1+y+2x+y^2}+\frac{y+2x}{1+y+2x+y^2}\\
& +\frac{1}{1+y+2x+y^2} +\frac{1+y+x+y^2}{1+y+2x+y^2}+\frac{1+y+x+yx}{1+y+2x+y^2}-\frac{-x-y^2+yx}{1+y+2x+y^2}\\
& -\frac{1+2y+2x}{1+y+2x+y^2} -\frac{x(-1+y)}{1+y+2x+y^2}+\frac{(y-x)(-1+y)}{1+y+2x+y^2},
\end{split}
\end{equation*}
where $Q= A P A P^{-1}$.

Considering the same map with the case\hspace{0.2cm} $t$\hspace{0.2cm} is even, we get the determinant of the Jacobian matrix at\hspace{0.2cm} $(2, 3)$\hspace{0.2cm} is\hspace{0.2cm} $\displaystyle -\frac{1792}{4913}$.
\vspace{0.5cm}

If\hspace{0.2cm} $(t-1)/2$\hspace{0.2cm} is odd then
\begin{equation*}
\begin{split} \text{trace}(Q) & =t-6-\frac{3y}{y^2+2y+2x}-\frac{-y-x}{y^2+2y+2x}-\frac{2y^2}{y^2+2y+2x}-\frac{y+x}{y^2+2y+2x}\\
& -\frac{-y^2+yx-y-x}{y^2+2y+2x} -\frac{xy+y+x}{y^2+2y+2x}+\frac{y^2+y+x}{y^2+2y+2x}+\frac{x}{y^2+2y+2x}\\
& + \frac{yx+2x+y}{y^2+2y+2x} -\frac{y(x-y-1)}{y^2+2y+2x} -\frac{xy}{y^2+2y+2x}+\frac{y(y-x)}{y^2+2y+2x}
\end{split}
\end{equation*}
and the determinant of the Jacobian matrix at\hspace{0.2cm} $(2, 3)$\hspace{0.2cm} is\hspace{0.2cm} $\displaystyle -\frac{768}{6859}$.

In each case the determinant of the Jacobian is nonzero and thus the image of the map\hspace{0.2cm} $f \circ g$\hspace{0.2cm} contains an open set.

\textbf{Case 3:} If\hspace{0.2cm} $A$\hspace{0.2cm} has more than two\hspace{0.2cm} $-1$\hspace{0.2cm} eigenvalues, we take\hspace{0.2cm} $P$\hspace{0.2cm} as below.

First, consider the following composition.
$$\mathbb{R}^k  \longrightarrow  GL(n+1, \mathbb{R}) \longrightarrow  SL(n+1, \mathbb{R})  \longrightarrow  \mathbb{R}^k,$$
given by
$$(x_1, x_2, ..., x_k)   \longmapsto  P  \longmapsto f(P)=A P A P^{-1}= Q  \longmapsto  g(Q),$$
where\hspace{0.2cm} $g(Q)= (\textrm{trace}(Q), \textrm{trace} (Q^2), ... , \textrm{trace}(Q^k))$\hspace{0.2cm} and\hspace{0.2cm} $k$\hspace{0.2cm} is the number of\hspace{0.2cm} $-1$\hspace{0.2cm} eigenvalues of\hspace{0.2cm} $A$.\hspace{0.2cm}
The Jacobian matrix is given by
\begin{displaymath}
\mathbf{J} =
\left[ \begin{array}{cccc}
\displaystyle \frac{\partial\ \textrm{trace}(Q)}{\partial x_1} & \displaystyle \frac{\partial\ \textrm{trace}(Q)}{\partial x_2} & ... &  \displaystyle \frac{\partial\ \textrm{trace}(Q)}{\partial x_k} \\
\displaystyle \frac{\partial\ \textrm{trace}(Q^2)}{\partial x_1} & \displaystyle \frac{\partial\ \textrm{trace}(Q^2)}{\partial x_2} & ... & \displaystyle \frac{\partial\ \textrm{trace}(Q^2)}{\partial x_k} \\
\vdots & \vdots & \vdots & \vdots \\
\displaystyle \frac{\partial\ \textrm{trace}(Q^k)}{\partial x_1} & \displaystyle \frac{\partial\ \textrm{trace}(Q^k)}{\partial x_2} & ... & \displaystyle \frac{\partial\ \textrm{trace}(Q^k)}{\partial x_k}
\end{array} \right].
\end{displaymath}

$\bullet$ If\hspace{0.2cm} $t$\hspace{0.2cm} is even,

let
$a_{12}=x_2,\quad a_{1(t/2)}=a_{1(t+2)/2}=x_3,\quad a_{1(t-1)}=x_1,\quad a_{2(t-2)}=x_3,\\ a_{(t/2)1}=x_2,\quad
 a_{((t+2)/2)1}=x_3,\quad a_{(t/2)t}=x_3,\quad a_{((t+2)/2)t}=x_1,\\ a_{(t-1)1}=1,\quad a_{t2}=x_3,\quad a_{t(t-1)}=x_2$,\hspace{0.2cm} and all the diagonal elements are\hspace{0.2cm} $1$.

According to the number of\hspace{0.2cm} $-1$\hspace{0.2cm} eigenvalues of\hspace{0.2cm} $A$,\hspace{0.2cm} we determine the number of different variables\hspace{0.2cm} $x_i \in \mathbb{R}$,\hspace{0.2cm} where\hspace{0.2cm} $3\leq i \leq k$\hspace{0.2cm} and\hspace{0.2cm} $k=t/2$. In the core matrix, on the antidiagonal there are only\hspace{0.2cm} $x_i$'s\hspace{0.2cm} (except $x_3$) as a pair, which are symmetric with respect to the diagonal. Moreover, the number of some\hspace{0.2cm} $x_i$'s\hspace{0.2cm} are more than two conforming to the dimension. In addition, other entries of\hspace{0.2cm} $P$\hspace{0.2cm} are all\hspace{0.2cm} $0$.

For example, if\hspace{0.2cm} $A$\hspace{0.2cm} has six\hspace{0.2cm} $-1$\hspace{0.2cm} eigenvalues and\hspace{0.2cm} $t=14$\hspace{0.2cm} then\hspace{0.2cm} $P$\hspace{0.2cm} is as below.\label{matris5}
\begin{displaymath}
{P} =
\left[ \begin{array}{cccccccccccccc}
1 & x_2 & 0 & 0 & 0 & 0 & x_3 & x_3 & 0 & 0 & 0 & 0 & x_1 & 0 \\
0 & 1 & 0 & 0 & 0 & 0 & 0 & 0 & 0 & 0 & 0 & x_3& x_2 & 0 \\
0 & 0 & 1 & 0 & 0 & 0 & 0 & 0 & 0 & 0 & 0 & x_2 & 0 & 0 \\
0 & 0 & 0 & 1 & 0 & 0 & 0 & 0 & 0 & 0 & x_5 & 0 & 0 & 0\\
0 & 0 & 0 & 0 & 1 & 0 & 0 & 0 & 0 & x_6 & 0 & 0 & 0 & 0\\
0 & 0 & 0 & 0 &0 & 1 & 0 & 0 & x_4 & 0 & 0 & 0 & 0 & 0\\
x_2 & 0 & 0 & 0 & 0 & 0 & 1 & x_1& 0 & 0 & 0 & 0 & 0 & x_3\\
x_3 & 0 & 0 & 0 & 0 & 0 & x_1 & 1 & 0 & 0 & 0 & 0 & 0 & x_1\\
0 & 0 & 0 & 0 & 0 & x_4 & 0 & 0 & 1 & 0 & 0 & 0 & 0 & 0 \\
0 & 0 & 0 & 0 & x_6 & 0 & 0 & 0 & 0 & 1 & 0 & 0 & 0 & 0\\
0 & 0 & 0 & x_5 & 0 & 0 & 0 & 0 & 0 & 0 & 1 & 0 & 0 & 0 \\
0 & 0 & x_2 & 0 & 0 & 0 & 0 & 0 & 0 & 0 & 0 & 1 & 0 & 0 \\
1 & x_2 & 0 & 0 & 0 & 0 & 0 & 0 & 0 & 0 & 0 & 0 & 1 & 0 \\
0 & x_3 & 0 & 0 & 0 & 0 & 0 & 0 & 0 & 0 & 0 & 0 & x_2 & 1
\end{array} \right].
\end{displaymath}
At the point \hspace{0.2cm}$(2, 3, 4, 5, 6, 7)$\hspace{0.2cm} the determinant of the Jacobian is $$\frac{3203652023}{129225403018523774123952000}.$$

$\bullet$ If\hspace{0.2cm} $t$\hspace{0.2cm} is odd,

let $a_{12}=x_2,\quad a_{1(t+1)/2}=x_3,\quad a_{1(t-1)}=x_1,\quad a_{2(t-2)}=x_3,\\ a_{((t-1)/2)1}=x_2,\quad a_{((t+1)/2)1}=1,\quad a_{((t+3)/2)1}=x_3,\quad a_{(t-1)1}=1,\\ a_{((t-1)/2)t}=x_3,\quad a_{((t+3)/2)t}=x_1,\quad a_{t2}=x_3,\quad a_{t(t-1)}=x_2$\hspace{0.2cm} and the diagonal elements are all\hspace{0.2cm} $1$.

In the core matrix, on the antidiagonal there are only\hspace{0.2cm} $x_i$'s\hspace{0.2cm} (except $x_3$) as a pair, which are symmetric with respect to the diagonal. Moreover, the number of some\hspace{0.2cm} $x_i$'s\hspace{0.2cm} are more than two conforming to the dimension. In addition, other entries of \hspace{0.2cm}$P$\hspace{0.2cm} are all\hspace{0.2cm} $0$.

For example, if\hspace{0.2cm} $A$\hspace{0.2cm} has five\hspace{0.2cm} $-1$\hspace{0.2cm} eigenvalues and\hspace{0.2cm} $t=13$\hspace{0.2cm} then\hspace{0.2cm} $P$\hspace{0.2cm} is as follows:\label{matris6}
\begin{displaymath}
{P} =
\left[ \begin{array}{ccccccccccccc}
1 & x_2 & 0 & 0 & 0 & 0 & x_3 & 0 & 0 & 0 & 0 & x_1 & 0 \\
0 & 1 & 0 & 0 & 0 & 0 & 0 & 0 & 0 & 0 & x_3& x_2 & 0 \\
0 & 0 & 1 & 0 & 0 & 0 & 0 & 0 & 0 & 0 & x_2 & 0 & 0 \\
0 & 0 & 0 & 1 & 0 & 0 & 0 & 0 & 0 & x_5 & 0 & 0 & 0\\
0 & 0 & 0 & 0 & 1 & 0 & 0 & 0 & x_4 & 0 & 0 & 0 & 0\\
x_2 & 0 & 0 & 0 & 0 & 1 & 0& x_1 & 0 & 0 & 0 & 0 & x_3\\
1 & 0 & 0 & 0 & 0 & 0 & 1 & 0 & 0 & 0 & 0 & 0 & 0\\
x_3 & 0 & 0 & 0 & 0 & x_1 & 0 & 1 & 0 & 0 & 0 & 0 & x_1\\
0 & 0 & 0 & 0 & x_4 & 0 & 0 & 0 & 1 & 0 & 0 & 0 & 0 \\
0 & 0 & 0 & x_5 & 0 & 0 & 0 & 0 & 0 & 1 & 0 & 0 & 0 \\
0 & 0 & x_2 & 0 & 0 & 0 & 0 & 0 & 0 & 0 & 1 & 0 & 0 \\
1 & x_2 & 0 & 0 & 0 & 0 & 0 & 0 & 0 & 0 & 0 & 1 & 0 \\
0 & x_3 & 0 & 0 & 0 & 0 & 0 & 0 & 0 & 0 & 0 & x_2 & 1
\end{array} \right].
\end{displaymath}
At the point\hspace{0.2cm} $(2, 3, 4, 5, 6)$\hspace{0.2cm} the determinant of the Jacobian is
$$\frac{74929536}{42961619719375}.$$

\textbf{Case 4:} If\hspace{0.2cm} $A$\hspace{0.2cm} has eigenvalues\hspace{0.2cm} $\pm i$\hspace{0.2cm} then both\hspace{0.2cm} $+i$\hspace{0.2cm} eigenspace and\hspace{0.2cm} $-i$\hspace{0.2cm} eigenspace of\hspace{0.2cm} $A$\hspace{0.2cm} are\hspace{0.2cm} $\displaystyle \frac{n+1}{2}$\hspace{0.2cm} dimensional. Now, we choose\hspace{0.2cm} $P$\hspace{0.2cm} as in Case 3 with\hspace{0.2cm} $k=\displaystyle \frac{n+1}{2}$ variables.
\vspace{0.1cm}

Note that the calculations above are done with the program Maple.
\end{appendix}

\bibliographystyle{plain} %unsrt,alpha,abbrv,plain

\bibliography{biblio}

\end{document}